\documentclass[twoside,11pt]{article}

\usepackage{jmlr2e}

\usepackage{times}
\usepackage{graphicx} 
\usepackage{subfigure} 

\usepackage{natbib}

\usepackage{algorithm}
\usepackage{algorithmic}

\usepackage{comment}
\usepackage{hyperref}

\usepackage{amsmath,amssymb}
\usepackage{amsbsy}
\usepackage{bm}

\newcommand{\argmin}{\mathop{\mathrm{argmin}}}

\newcommand{\Eqref}[1]{Eq. \eqref{#1}}

\newcommand{\boldone}{{\boldsymbol{1}}}

\newcommand{\calA}{{\mathcal A}}
\newcommand{\calB}{{\mathcal B}}

\newcommand{\calF}{{\mathcal F}}
\newcommand{\calG}{{\mathcal G}}

\newcommand{\calQ}{{\mathcal Q}}

\newcommand{\calT}{{\mathcal T}}

\newcommand{\calH}{{\mathcal H}}
\newcommand{\calX}{{\mathcal X}}

\newcommand{\fhat}{\widehat{f}}

\newcommand{\fbar}{\bar{f}}

\newcommand{\LPi}{L_2} 

\newcommand{\Real}{\mathbb{R}}

\newcommand{\EE}{\mathrm{E}}
\newcommand{\dd}{\mathrm{d}}
\newcommand{\fstar}{f^{\ast}}

\newcommand{\ftrue}{f^{\mathrm{o}}}
\newcommand{\htrue}{h^{\mathrm{o}}}
\newcommand{\btrue}{b^{\mathrm{o}}}
\newcommand{\Ftrue}{F^{\mathrm{o}}}

\def\I<#1>{\left\langle #1 \right\rangle}
\def\i<#1>{\left\langle #1 \right\rangle}

\allowdisplaybreaks
\newtheorem{Theorem}{Theorem}
\newtheorem{Lemma}{Lemma}

\newtheorem{Assumption}{Assumption}

\newtheorem{Proposition}{Proposition}

\newtheorem{Corollary}{Corollary}

\newcommand{\Ltwo}{L_2}
\renewcommand{\LPi}{L_2(P(X))}
\newcommand{\dx}{d_{\mathrm{x}}}

\newcommand{\kell}{\mathsf{k}_\ell}

\newcommand{\Tr}{\mathrm{Tr}}
\newcommand{\F}{\mathrm{F}}
\newcommand{\mell}{m_\ell}
\newcommand{\czero}{c_0}
\newcommand{\cone}{c_1}
\newcommand{\conedelta}{\hat{c}_{\delta}}

\newcommand{\Well}[1]{W^{(#1)}}
\newcommand{\bell}[1]{b^{(#1)}}

\newcommand{\Wstarell}[1]{W^{*(#1)}}
\newcommand{\bstarell}[1]{b^{*(#1)}}

\newcommand{\htrueell}[1]{h^{\circ}_{#1}}

\newcommand{\cdelta}{c_\delta}

\newcommand{\Rb}{R_b}

\newcommand{\epsilonn}{\epsilon_n}

\newcommand{\logone}{\log_+}

\newcommand{\Rhatinf}{\hat{R}_{\infty}}

\newcommand{\rtil}{\tilde{r}}

\newcommand{\Rbar}{\bar{R}}
\newcommand{\Rbarb}{\bar{R}_b}

\newcommand{\Unif}{\mathrm{U}}

\newcommand{\deltanone}{\hat{\delta}_{1,n}}
\newcommand{\deltantwo}{\hat{\delta}_{2,n}}

\newcommand{\tilepsilonn}{\tilde{\epsilon}_n}

\newcommand{\fstarsub}{\tilde{f}^*}

\ShortHeadings{Fast Learning Rate of Deep Learning}{T. Suzuki}
\firstpageno{1}

\begin{document}

\title{Fast learning rate of deep learning via a kernel perspective}

\author{\name Taiji Suzuki \email taiji@mist.i.u-tokyo.ac.jp \\
       \addr Department of Mathematical Informatics \\
       The University of Tokyo  \\
       7-3-1 Hongo, Bunkyo-ku, Tokyo 113-8656, Japan, \\
\vspace{-0.3cm} \\
       PRESTO, Japan Science and Technology Agency, \\
\vspace{-0.3cm} \\
	Center for Advanced Integrated Intelligence Research, RIKEN}


\maketitle


\title{Fast generalization error bound of deep learning: \\
a kernel perspective}

\maketitle 

\begin{abstract}
We develop a new theoretical framework 
to analyze the generalization error of deep learning,
and derive a new fast learning rate for two representative algorithms: 
{\it empirical risk minimization} and {\it Bayesian deep learning}.
The series of theoretical analyses of deep learning 
has revealed its high expressive power and universal approximation capability.
Although these analyses are highly nonparametric, 
existing generalization error analyses have been developed mainly in a fixed dimensional parametric model.
To compensate this gap, we develop an infinite dimensional model 
that is based on an integral form as performed in the analysis of the universal approximation capability.
This allows us to define a reproducing kernel Hilbert space corresponding to each layer.
Our point of view is to deal with the ordinary finite dimensional deep neural network 
as a finite approximation of the infinite dimensional one.
The approximation error is evaluated by the {\it degree of freedom} of the reproducing kernel Hilbert space in each layer.
To estimate a good finite dimensional model, 
we consider both of empirical risk minimization and Bayesian deep learning.
We derive its generalization error bound
and it is shown that there appears bias-variance trade-off 
in terms of the number of parameters of the finite dimensional approximation.
We show that the optimal width of the internal layers can be determined through the degree of freedom 
and the convergence rate can be faster than $O(1/\sqrt{n})$ rate which has been shown in the existing studies.


\end{abstract}


\section{Introduction}

Deep learning has been showing great success 
in several applications such as 
computer vision, natural language processing,
and many other area related to pattern recognition.
Several high-performance methods have been developed and 
it has been revealed that deep learning possesses great potential.
Despite the development of practical methodologies,
its theoretical understanding is not satisfactory.
Wide rage of researchers including theoreticians and practitioners
are expecting deeper understanding of deep learning.

Among theories of deep learning, 
a well developed topic is its expressive power.
It has been theoretically shown that 
deep neural network has exponentially large expressive power against the number of layers.
For example, \citet{NIPS2014_5422}
showed that the number of polyhedral regions created by deep neural network 
can exponentially grow as the number of layers increases.
\citet{bianchini2014complexity} showed that the Betti numbers of the level set of a function created by deep neural network
grows up exponentially against the number of layers.
Other researches 
also concluded similar facts using different notions such as tensor rank and extrinsic curvature
\citep{cohen2016expressive,ICML:Cohen+Shashua:2016,NIPS:Poole+etal:2016}.

Another important issue in neural network theories is its universal approximation capability.
It is well known that 3-layer neural networks have the ability,
and thus the deep neural network also does \citep{cybenko1989approximation,hornik1991approximation,sonoda2015neural}.
When we discuss the universal approximation capability, 
the target function that is approximated is arbitrary and 
the theory is highly nonparametric in its nature.

Once we knew the expressive power and universal approximation capability of deep neural network,
the next theoretical question naturally arises in its generalization error.
The generalization ability is typically analyzed by evaluating the {\it Rademacher complexity}. 
\citet{bartlett1998sample} studied 3-layer neural networks and characterized its Rademacher complexity using the norm of weights. 
\citet{koltchinskii2002empirical} studied deep neural network and derived its Rademacher complexity under norm constraints.
More recently, \citet{COLT:Neyshabur+Tomioka+Srebro:2015}
analyzed the Rademacher complexity based on more generalized norm, and 
\citet{sun2015large} derived a generalization error bound  
with a large margine assumption.
As a whole, the studies listed above derived $O(1/\sqrt{n})$ convergence of the generalization error where $n$ is the sample size.
One concern in this line of convergence analyses is that 
the convergence of the generalization error is only $O(1/\sqrt{n})$ where $n$ is the sample size.  
Although this is minimax optimal, it is expected that we could show faster convergence rate 
with some additional assumptions such as strong convexity of the loss function.
Actually, in a regular parametric model, we have $O(1/n)$ convergence of the generalization error \citep{hartigan1998maximum}.
%
%
Moreover, 
the generalization error bound has been mainly given in finite dimensional models.
As we have observed, the deep neural network 
possesses exponential expressive power and universal approximation capability
which are highly nonparametric characterizations.
This means that the theories are developed separately in the two regimes; finite dimensional parametric model and infinite dimensional nonparametric model.
Therefore, theories that connect these two regimes are expected to comprehensively understand statistical performance of deep learning.

In this paper,
we consider both of empirical risk minimization and Bayesian deep learning and 
analyze the generalization error using the terminology of kernel methods.
Consequently, (i) we derive a faster learning rate than $O(1/\sqrt{n})$
and (ii) we connect the finite dimensional regime and the infinite dimensional regime based on the theories of kernel methods.
The empirical risk minimization is a typical approach to learn the deep neural network model.
It is usually performed by applying stochastic gradient descent with the back-propagation technique \citep{widrow1960adaptive,amari1967theory,rumelhart1988learning}.
To avoid over-fitting, such techniques as regularization and dropout have been employed \citep{srivastava2014dropout}.
Although the practical techniques for the empirical risk minimization have been extensively studied, 
there is still much room for improvement in its generalization error analysis. 
Bayesian deep learning has been recently gathering more attentions mainly because it can deal with the estimation uncertainty in a natural way.
Examples of Bayesian deep learning researches include
probabilistic backpropagation \citep{hernandez2015probabilistic},
Bayesian dark knowledge \citep{balan2015bayesian},
weight uncertainty by Bayesian backpropagation \citep{blundell2015weight}, 
dropout as Bayesian approximation \citep{Gal+Ghahramani2016}.
To analyze a sharper generalization error bound, we utilize the so-called local Rademacher complexity technique for the empirical risk minimization method 
\citep{IEEEIT:Mendelson:2002,LocalRademacher,Koltchinskii,gine2006concentration},
and, as for the Bayesian method, we employ the theoretical techniques developed to analyze nonparametric Bayes methods
\citep{AS:Ghosal+Ghosh+Vaart:2000,AS:Vaart&Zanten:2008,JMLR:Vaart&Zanten:2011}. 
These analyses are quite advantageous to the typical Rademacher complexity analysis because 
we can obtain convergence rate between $O(1/n)$ and $O(1/\sqrt{n})$ which is faster than that of the standard Rademacher complexity analysis $O(1/\sqrt{n})$.
As for the second contribution, we first introduce an integral form of deep neural network
as performed in the research of the universal approximation capability of 3-layer neural networks \citep{sonoda2015neural}.
This allows us to have a nonparametric model of deep neural network
 as a natural extension of usual finite dimensional models.
Afterward, we define a reproducing kernel Hilbert space (RKHS) corresponding to each layer
like in \cite{JMLR:v18:14-546,bach2015equivalence}.
By doing so, we can borrow the terminology developed in the kernel method into the analysis of deep learning.
In particular, we define the degree of freedom of the RKHS as a measure of complexity of the RKHS \citep{FCM:Caponetto+Vito:2007,bach2015equivalence}, and
based on that, we evaluate how large a finite dimensional model should be to approximate the original infinite dimensional model with a specified precision. 
These theoretical developments reveal that 
there appears bias-variance trade-off.
That is, there appears trade-off between 
the size of the finite dimensional model approximating the nonparametric model 
and the variance of the estimator.
We will show that, by balancing the trade-off, 
a fast convergence rate is derived.
In particularly, the optimal learning rate of the kernel method is reproduced 
from our deep learning analysis due to the fact that 
the kernel method can be seen as a 3-layer neural network with an infinite dimensional internal layer.
A remarkable property of the derived generalization error bound is that 
the error is characterized by the complexities of the RKHSs defined by the degree of freedom.
Moreover, the notion of the degree of freedom gives a practical implication about determination of the width of the internal layers.

The obtained generalization error bound is summarized in Table \ref{tab:result_summary}\footnote{$a \vee b$ indicates $\max\{a,b\}$.}.

\begin{table}
\begin{center}
\caption{Summary of derived bounds for the generalization error $\|\fhat - \ftrue \|^2_{\LPi}$
where $n$ is the sample size, 
$R$ is the norm of the weight in the internal layers, 
$\Rhatinf$ is an $L_\infty$-norm bound of the functions in the model,
$\sigma$ is the observation noise, $\dx$ is the dimension of the input,
$\mell$ is the width of the $\ell$-th internal layer
and $N_\ell(\lambda_\ell)$ for ($\lambda_\ell >0$) is the degree of freedom (\Eqref{eq:DefDOF}).}
\label{tab:result_summary}
\begin{tabular}{|c|c|}
\hline
 & Error bound  \\
\hline 
General setting & 
$L \sum_{\ell=2}^L R^{L-\ell +1} \lambda_\ell + \frac{\sigma^2 + \Rhatinf^2}{n}\sum_{\ell=1}^L m_{\ell} m_{\ell+1} \log(n)$ 
\\
& under an assumption that $\mell \gtrsim N_\ell(\lambda_\ell) \log(N_\ell(\lambda_\ell))$.
 \\
\hline 
Finite dimensional model & 
$\frac{\sigma^2 + \Rhatinf^2}{n}\sum_{\ell=1}^L m^*_{\ell} m^*_{\ell+1} \log(n)$ \\
 & where $m_\ell^*$ is the true width of the $\ell$-th internal layer.
 \\
\hline 
Polynomial decay eigenvalue& 
$
L \sum_{\ell=2}^L (R\vee 1)^{L - \ell + 1} n^{-\frac{1}{1+2s_\ell}} \log(n) 
+ \frac{\dx^2}{n}\log(n)
$ \\
& where $s_\ell$ is the decay rate of the eigenvalue of the kernel  \\ 
& function on the $\ell$-th layer.
 \\
\hline 
\end{tabular}
\end{center}
\end{table}

\section{Integral representation of deep neural network}

Here we give our problem settings and the model that we consider in this paper.
Suppose that 
$n$ input-output observations $D_n = (x_i,y_i)_{i=1}^n \subset \Real^{\dx} \times \Real$
are independently identically 
generated 
from a regression model 
$$
y_i = \ftrue(x_i) + \xi_i~~(i=1,\dots,n)
$$
where $(\xi_i)_{i=1}^n$ is an i.i.d. sequence of Gaussian noises $N(0,\sigma^2)$ with mean 0 and variance $\sigma^2$,
and $(x_i)_{i=1}^n$ is generated independently identically from a distribution $P(X)$ with a compact support in $\Real^{\dx}$.  
The purpose of the deep learning problem we consider in this paper is 
to estimate $\ftrue$ from the $n$ observations $D_n$.

To analyze the generalization ability of deep learning,
we specify a function class in which the true function $\ftrue$ is included,
and, by doing so, we characterize the ``complexity'' of the true function in a correct way.

In order to give a better intuition, we first start from the simplest model, the 3-layer neural network.
Let $\eta$ be a nonlinear activation function such as ReLU \citep{nair2010rectified,glorot2011deep};
$\eta(x) = (\max\{x_i,0\})_{i=1}^d$ for a $d$-dimensional vector $x \in \Real^d$.
The 3-layer neural network model is represented by 
$$
f(x) = \Well{2} \eta ( \Well{1} x + \bell{1}) + \bell{2}
$$
where we denote by $m_2$ the number of nodes in the internal layer, and 
$\Well{2} \in \Real^{1 \times m_2}$, $\Well{1} \in \Real^{m_2 \times \dx}$, $\bell{1} \in \Real^{m_2}$ and 
$\bell{2} \in \Real$.
It is known that this model is {\it universal approximator} and it is important to consider its {\it integral form}
\begin{align}
f(x) = \int h(w,b) \eta (w^\top x + b) \dd w \dd b + \bell{2}.
\label{eq:integralformFirst}
\end{align}
where $(w,b) \in \Real^{\dx} \times \Real$ is a hidden parameter, $h:\Real^{\dx} \times \Real \to \Real$ is a function version of 
the weight matrix $\Well{2}$, and $\bell{2} \in \Real$  is the bias term.
This integral form appears in many places to analyze the capacity of the neural network.
In particular, through the ridgelet analysis, it is shown that there exists the integral form corresponding to any $f \in L_1(\Real^{\dx})$
which has an integrable Fourier transform for an appropriately chosen activation function $\eta$ such as ReLU \citep{sonoda2015neural}.

Motivated by the integral form of the 3-layer neural network,
we consider a more general representation for deeper neural network.
To do so, we define a feature space on the $\ell$-th layer.
The feature space is a 
a probability space 
$
(\calT_\ell, \calB_\ell,\calQ_\ell)
$
where $\calT_\ell$ is a Polish space, $\calB_\ell$ is its Borel algebra, and $\calQ_\ell$ is a probability measure on $(\calT_\ell,\calB_\ell)$.
This is introduced to represent a general (possibly) continuous set of features as well as 
a discrete set of features.
For example, if the $\ell$-th internal layer is endowed with a $d_\ell$-dimensional finite feature space,
then $\calT_{\ell} = \{1,\dots,d_\ell\}$.
On the other hand,
the integral form \eqref{eq:integralformFirst} corresponds to a continuous feature space $\calT_2 = \{(w,b) \in \Real^{\dx} \times \Real\}$
in the second layer.
Now the input $x$ is a $\dx$-dimensional real vector, and thus we may set $\calT_1 = \{1,\dots,\dx\}$.
Since the output is one dimensional, the output layer is just a singleton $\calT_{L+1} = \{1\}$.
Based on these feature spaces, our integral form of the deep neural network 
is constructed by stacking 
the map on the $\ell$-th layer 
$\ftrue_\ell: \Ltwo (Q_\ell) \to\Ltwo(Q_{\ell + 1}) $ 
given as 
\begin{subequations}
\begin{align}
& \ftrue_\ell [g] (\tau) = \int_{\calT_\ell} \htrue_\ell(\tau, w) \eta ( g(w) ) \dd Q_\ell(w) + \btrue_\ell (\tau),
\end{align}
where 
$\htrue_\ell(\tau,w)$ corresponds to the weight of the feature $w$ 
for the output $\tau$
and 
$\htrue_\ell \in \Ltwo(Q_{\ell+1} \times Q_{\ell})$ and $\htrue_\ell(\tau,\cdot) \in  \Ltwo(Q_{\ell+1})$ for all $\tau \in \calT_{\ell +1}$
\footnote{
Note that, for $g \in \Ltwo(Q_{\ell})$, 
$f_\ell[g]$ is also square integrable with respect to $\Ltwo(Q_{\ell +1})$
if $\eta$ is Lipschitz continuous 
because $h \in \Ltwo(Q_{\ell+1} \times Q_{\ell})$.}.
Specifically, 
the first and the last layers are represented as 
\begin{align}
& \ftrue_1 [x] (\tau) = \sum_{j=1}^{\dx} \htrue_1(\tau,j) x_j  Q_1(j) + \btrue_1 (\tau), \\
& \ftrue_L [g] (1) = \int_{\calT_L} \htrue_L (w) \eta ( g(w) ) \dd Q_L(w) + \btrue_L,
\end{align} \label{eq:flIntegralForm}
\end{subequations} 
where we wrote $\htrue_L(w)$ to indicate $\htrue_L(1,w)$ for simplicity because $\calT_{L+1} = \{1\}$.
Then the true function $\ftrue$ is given as 
\begin{align}
\ftrue(x) = \ftrue_L \circ \ftrue_{L-1} \circ \dots \circ \ftrue_{1} (x).
\label{eq:flstacked}
\end{align}
Since, the shallow 3-layer neural network is a universal approximator,
and so is our generalized deep neural network model \eqref{eq:flstacked}.
It is known that deep neural network tends to give more efficient representation 
of a function than the shallow network.
Actually, \citet{eldan2016power} gave an example of a function that
the 3-layer neural network cannot approximate under a precision unless its with is exponential in the input dimension
but the 4-layer neural network can approximate with polynomial order widths
(see \citet{safran2016depth} for other examples).
In other words, each layer of a deep neural network can be much ``simpler'' 
than one of a shallow network (more rigorous definition of complexity of each layer will be given in the next section).
Therefore, it is quite important to consider the integral representation of a deep neural network 
rather than a 3-layer network.


The integral representation is natural also from the practical point of view.
Indeed, it is well known that the deep neural network learns 
a simple pattern in the early layers and it gradually extracts more complicated features 
as the layer is going up.
The trained feature is usually continuous one.
For example, in computer vision tasks, 
the second layer typically extracts gradients toward several degree angles \citep{krizhevsky2012imagenet}.
The angle is a continuous variable and thus the feature space should be continuous to cover all angles. 
On the other hand, the real network discretize the feature space because of limitation of computational resources.
Our theory introduced in the next section offers a measure to evaluate this discretization error.

\section{Finite approximation of the integral form}
\label{sec:FiniteApproximation}

The integral form is a convenient way to describe the true function.
However, it is not useful to estimate the function.
When we estimate that, we need to discretize the integrals by finite sums due to limitation of computational resources
as we do in practice.
In other word, we consider the usual finite sum deep learning model as an approximation of 
the integral form.
However, the discrete approximation induces approximation error.
Here we give an upper bound of the approximation error.
Naturally, there arises the notion of bias and variance trade-off,
that is, as the complexity of the finite model increases the ``bias'' (approximation error) decreases
but the ``variance'' for finding the best parameter in the model increases.
Afterwards, we will bound the variance for estimating the finite approximation in Section \ref{sec:GenErrorBound}.
Combining these two notions, 
it is possible to quantify the bias-variance trade-off and find the best strategy to minimize the entire generalization error.

The approximation error analysis of the deep neural network can be 
well executed by utilizing notions of the kernel method.
Here we construct RKHS for each layer 
in a way analogous to \cite{bach2015equivalence,JMLR:v18:14-546} 
who studied shallow learning and the kernel quadrature rule.
Let the output of the $\ell$-th layer be $\Ftrue_\ell(x,\tau) := (\ftrue_{\ell} \circ \dots \circ \ftrue_{1} (x))(\tau).$
We define a {\it reproducing kernel Hilbert space} (RKHS) corresponding to the $\ell$-th layer ($\ell \geq 2$) by introducing its 
associated kernel function $\kell: \Real^{\dx} \times \Real^{\dx} \to \Real$.
We define the positive definite kernel $\kell$ as
$$
\kell(x,x') := \int_{\calT_\ell}  \eta (\Ftrue_{\ell-1 }(x,\tau))  \eta (\Ftrue_{\ell-1 }(x',\tau)) \dd Q_\ell (\tau).
$$
It is easy to check that $\kell$ is actually symmetric and positive definite.
It is known that there exists a unique RKHS $\calH_\ell$ corresponding the kernel $\kell$ \citep{AMS:Aronszajn:1950}.
Close investigation of the RKHS for several examples for shallow network has been given in \citep{JMLR:v18:14-546}.

Under this setting, all arguments at the $\ell$-th layer can be carried out through the theories of kernel methods.
Importantly, for $g \in \calH_\ell$, there exists $h \in \Ltwo(Q_\ell)$ such that 
$$
g(x) = \int_{T_\ell} h(\tau) \eta (\Ftrue_{\ell-1}(x,\tau)) \dd Q_\ell(\tau).
$$
Moreover, the norms of $g$ and $h$ are connected as 
\begin{align}
\|g\|_{\calH_\ell} = \|h\|_{\Ltwo(Q_\ell)},
\label{eq:ghnormeq}
\end{align}
\citep{bach2015equivalence,JMLR:v18:14-546}.
Therefore, the function 
$$
x \mapsto 
\int_{\calT_\ell} \htrue_\ell(\tau, w) \eta ( \Ftrue_{\ell-1}(x,w) ) \dd Q_\ell(w),
$$
representing the magnitude of a feature $\tau \in \calT_{\ell + 1}$
for the input $x$
is included in the RKHS 
and its RKHS norm is equivalent to that of the internal layer weight  $\|\htrue(\tau,\cdot)\|_{\Ltwo(Q_\ell)}$ 
because of \Eqref{eq:ghnormeq}.

To derive the approximation error, we need to evaluate the ``complexity'' of the RKHS.
Basically, the complexity of the $\ell$-th layer RKHS $\calH_\ell$ is controlled by 
the behavior of the eigenvalues of the kernel.
To formally state this notion, 
we introduce the integral operator associated with the kernel $\kell$ defined as 
\begin{align*}
T_\ell: 
& g \mapsto \int_\calX \kell(\cdot, x) g(x) \dd P(x), \\
& \Ltwo(P(X)) \to \Ltwo(P(X)).
\end{align*}
If the kernel function admits an orthogonal decomposition 
$$
\kell(x,x') = \sum_{j=1}^\infty \mu_j^{(\ell)} \phi_j^{(\ell)}(x)\phi_j^{(\ell)}(x'),
$$
in $L_2(P(X) \times P(X))$ where $(\mu_j^{(\ell)})_{j=1}^\infty$ is the sequence of the eigenvalues 
ordered in decreasing order, 
and 
$(\phi_j^{(\ell)})_{j=1}^\infty$
forms an orthonormal system in $\Ltwo(P(X))$, 
then for $g(x) = \sum_{j=1}^\infty \alpha_j \phi_j^{(\ell)}(x)$,
the integral operation is expressed as $T_\ell g = \sum_{j=1}^\infty \alpha_j  \mu_j^{(\ell)} \phi_j^{(\ell)}$
(see \citet{Book:Steinwart:2008,ConstApp:Steinwart:2012} for more details).
Therefore each eigenvalue $\mu_j^{(\ell)}$ plays a role like a ``filter'' for each component $\phi_j^{(\ell)}$.
Here it is known that for all $g \in \calH_\ell$,
there exists $\bar{h} \in \Ltwo(P(X))$ such that $g = T_\ell \bar{h}$
and $\|g\|_{\calH_\ell} = \|\bar{h}\|_{\Ltwo(P(X))}$ \citep{FCM:Caponetto+Vito:2007,COLT:Steinwart+etal:2009}.
Combining this with \Eqref{eq:ghnormeq}, we have $\|g\|_{\calH_\ell} = \|h\|_{\Ltwo(Q_\ell)} = \|\bar{h}\|_{\Ltwo(P(X))}$

Based on the integral operator $T_\ell$, 
we define the {\it degree of freedom} $N_\ell(\lambda)$ of the RKHS as 
\begin{align}
N_\ell(\lambda) = \Tr [(T_\ell + \lambda )^{-1} T_\ell]
\label{eq:DefDOF}
\end{align}
for $\lambda > 0$.
The degree of freedom can be represented as $N_\ell(\lambda) = \sum_{j=1}^\infty \frac{\mu_j^{(\ell)}}{\mu_j^{(\ell)} + \lambda}$
by using the eigenvalues of the kernel.

Now, we assume that the true function $\ftrue$ satisfies a norm condition as follows.
\begin{Assumption}
\label{ass:hbnormbounds}
For each $\ell$, $\htrue_\ell$ and $\btrue_\ell$ satisfy that
\begin{align*}
& \|\htrue_\ell(\tau, \cdot)\|_{\Ltwo(Q_\ell)} \leq R~~(\forall \tau \in T_\ell), \\
& |\btrue_\ell(\tau)| \leq \Rb~~(\forall \tau \in T_\ell).
\end{align*}
\end{Assumption}
By \Eqref{eq:ghnormeq}, the first assumption $\|\htrue_\ell(\tau, \cdot)\|_{\Ltwo(Q_\ell)} \leq R$ is interpreted as 
$\Ftrue_\ell(\tau,\cdot) \in \calH_\ell$ and $\|\Ftrue_\ell(\tau,\cdot) \|_{\calH_\ell} \leq R$.
This means that the feature map $\Ftrue_\ell(\tau,\cdot) $ in each internal layer is well regulated by the RKHS norm.

Moreover, we also assume that the activation function is scale invariant.
\begin{Assumption}
\label{ass:EtaCondition}
We assume the following conditions on the activation function $\eta$.
\begin{itemize}
\setlength{\itemsep}{0pt}
\item $\eta$ is scale invariant: $\eta(a x) = a\eta(x)$ for all $a >0$ and $x\in \Real^d$ (for arbitrary $d$).
\item $\eta$ is 1-Lipschitz continuous: $|\eta(x) - \eta(x')| \leq \|x - x'\|$ for all $x,x'\in \Real^d$.
\end{itemize}
\end{Assumption}
The first assumption on the scale invariance is essential to derive tight error bounds.
The second one 
ensures that deviation in each layer does not affect the output so much.
The most important example of an activation function that satisfies these conditions is ReLU activation.
Another one is the identity map $\eta(x) = x$.

Finally we assume that the input distribution has a compact support. 
\begin{Assumption}
\label{ass:xbounds}
The support of $P(X)$ is compact and it is bounded as  
$$
\|x\|_\infty := \max_{1 \leq i \leq \dx} |x_i| \leq D_x~~(\forall x \in \mathrm{supp}(P(X))).
$$
\end{Assumption}

We consider a finite dimensional approximation $\fstar$ 
given as follows:
let $\mell$ be the number of nodes in the $\ell$-th internal layer
(we set the dimensions of the output and input layers to $m_{L+1} = 1$ and $m_1 = \dx$)
and consider a model
\begin{align*}
& \fstar_\ell (g) = \Well{\ell} \eta(g)  + \bell{\ell}~~(g \in \Real^{\mell},~\ell = 2,\dots,L), \\
& \fstar_1 (x) = \Well{1} x + \bell{1}, \\
& \fstar (x) = \fstar_L \circ \fstar_{L-1} \circ \dots \circ \fstar_1 (x),
\end{align*}
where $\Well{\ell} \in \Real^{m_{\ell+1} \times m_{\ell}}$ and $\bell{\ell} \in \Real^{m_{\ell + 1}}$.

\begin{Theorem}[Finite approximation error bound of the nonparametric model]
\label{th:FiniteApprox}

For any $1 > \delta > 0$ and $\lambda_\ell >0$, suppose that 
$$
m_\ell \geq 5 N_\ell(\lambda_\ell) \log\left(32 N_\ell(\lambda_\ell)/\delta \right)~~~
(\ell = 2,\dots,L),
$$ 
then there exist
$
\Well{\ell} \in \Real^{m_{\ell + 1} \times m_\ell}
$
and
$
\bell{\ell} \in \Real^{m_{\ell + 1}}
$
such that, by letting $\conedelta = \frac{4}{1-\delta}$,
\begin{subequations}
\label{eq:Wellbellnormbound}
\begin{align}
&\|W^{(\ell)}\|_{\F}^2 \leq \conedelta  R^2~~(\ell = 1,\dots,L), \\
&
\|\bell{\ell}\| \leq \Rb/(1-\delta)~~(\ell = 1,\dots,L),
\end{align}
\end{subequations}
and 
\begin{align}
& \|\ftrue - \fstar\|_{\Ltwo(P(X))}
\leq  
\sum_{\ell = 2}^L 2 \sqrt{ \conedelta^{L-\ell} }  R^{L-\ell + 1} \sqrt{\lambda_\ell},
\label{eq:Approxbound}  \\
& \|\fstar\|_\infty
\leq (\sqrt{\conedelta}R)^{L} D_x + \sum_{\ell = 1}^L (\sqrt{\conedelta}R)^{L-\ell} {\textstyle \frac{\Rb}{1-\delta}}.
\end{align}

\end{Theorem}

The proof is given in Appendix \ref{app:FiniteApprox}.
This theorem is proven by borrowing the theoretical technique recently developed for the {\it kernel quadrature rule} \citep{bach2015equivalence}.
We also employed some techniques analogous to the analysis of the low rank tensor estimation \citep{ICML:Suzuki:2015,ICML:Kanagawa+etal:2016,suzuki2016minimax}.
Intuitively, the degree of freedom $N_\ell(\lambda_\ell)$ is the intrinsic dimension of the $\ell$-th layer to achieve 
the $\sqrt{\lambda_\ell}$ approximation error. 
Indeed, we show in the proof that the $\ell$-th layer is approximated 
by the $m_\ell$ dimensional nodes with the precision $\sqrt{\lambda_\ell}$ 
under the condition $m_\ell = \Omega(N_\ell(\lambda_\ell) \log(N_\ell(\lambda_\ell)))$.
The error bound \eqref{eq:Approxbound} indicates that 
the total approximation error of the whole network is basically obtained by 
summing up the approximation error $\sqrt{\lambda_\ell}$ of each layer
where the factor $\sqrt{ \conedelta^{L-\ell} }  R^{L-\ell + 1}$ is a 
Lipschitz constant for error propagation.

We would like to emphasize that 
the approximation error bound \eqref{eq:Approxbound}
and 
the norm bounds \eqref{eq:Wellbellnormbound} of $\Well{\ell}$ and $\bell{\ell}$
are independent of the dimensions $(\mell)_{\ell=1}^L$ of the internal layers.
This is due to the scale invariance property of the activation function.
This is quite beneficial to derive a tight generalization error bound.
Indeed, without the scale invariance, 
we only have a much looser bound
$
\|\ftrue - \fstar\|_{\Ltwo(P(X))}
\leq  
\sum_{\ell = 2}^L 2  \sqrt{ m_{\ell +1} \conedelta^{L-\ell} }  R^{L-\ell + 1} \sqrt{\lambda_\ell}$, 
and $\|W^{(\ell)}\|_{\F}^2 \leq m_{\ell + 1} \conedelta  R^2 $, $\|\bell{\ell}\|^2 \leq m_{\ell + 1}   \Rb^2 $
which depend on the dimensions $(m_{\ell})_{\ell=1}^L$ and could be huge for small $\lambda_\ell$.
This would support the practical success of using the ReLU activation.

Let the norm bounds shown in Theorem \ref{th:FiniteApprox} be 
$$\Rbar = \sqrt{\conedelta}R ,~~\Rbarb = \Rb/(1-\delta).$$
Remind that Theorem \ref{th:FiniteApprox} gives an upper bound of the infinity norm of $\fstar$, that is,
$\|\fstar\|_\infty \leq \Rhatinf$ where 
$$
\Rhatinf = 
\Rbar^L D_x + \sum_{\ell = 1}^L \Rbar^{L - \ell} \Rbarb.
$$
Let the set of finite dimensional functions with the norm constraint \eqref{eq:Wellbellnormbound} be 
$$
\calF = \{f(x) = (\Well{L} \eta( \cdot) + \bell{L}) \circ \dots 
\circ (\Well{1} x + \bell{1}) \mid \| \Well{\ell}\|_{\F} \leq \Rbar,~\|\bell{\ell}\| \leq \Rbarb~(\ell = 1,\dots,L)\}.
$$
Then, we can show that the infinity norm of $\calF$ is also uniformly bounded as the following lemma.
\begin{Lemma}
\label{eq:fbounded}
For all $f \in \calF$, it holds that 
$\|f\|_\infty 
\leq \Rhatinf.$
\end{Lemma}
The proof is given in Appendix \ref{app:LinftyNormBound}. 
Because of this, we can derive the generalization error bound with respect to the population $L_2$-norm 
instead of the empirical $L_2$-norm. 
One can check that $\|f\|_\infty \leq \Rhatinf$ for all $f \in \calF$ by Lemma \ref{eq:fbounded} or Lemma \ref{supplemma:SupnormBounds}.

\section{Generalization error bounds}

In this section, we define the two estimators in the finite dimensional model introduced in the last section:
the empirical risk minimizer and the Bayes estimator.
The generalization error bounds for both of these estimators are derived.
We also give some examples in which the generalization error  
is analyzed in details.

\subsection{Notations}
Before we state the generalization error bounds, we prepare some notations.
Let 
$
\hat{G} = L \Rbar^{L-1} D_x + \sum_{\ell = 1}^L \Rbar^{L - \ell},
$
and define $\deltanone$, $\deltantwo$ 
as\footnote{We define $\logone(x) = \max\{1,\log(x)\}$.} 
\begin{align*}
\deltanone & = \sum_{\ell = 2}^L 2 \sqrt{ \conedelta^{L-\ell} }  R^{L-\ell + 1} \sqrt{\lambda_\ell}, \\
\deltantwo^2 & =
\frac{2}{n} \sum_{\ell=1}^L \mell m_{\ell +1}
\logone \left(
{\textstyle 1 + \frac{4 \sqrt{2} \hat{G}\max\{\Rbar,\Rbarb\}
\sqrt{n}}{
\sigma \sqrt{\sum_{\ell=1}^L \mell m_{\ell +1}}} 
}
\right). 
\end{align*}
Note that $\deltanone$ is the finite approximation error given in Theorem \ref{th:FiniteApprox}.
Roughly speaking, $\deltantwo$ corresponds to the amount of deviation of 
the estimators in the finite dimensional model.

\subsection{Empirical risk minimization}
\label{sec:EmpiricalRiskMin}

In this section, we define the empirical risk minimizer and investigate its generalization error.
Let the empirical risk minimizer be $\fhat$:
$$
\fhat := \argmin_{f \in \calF} \sum_{i=1}^n (y_i - f(x_i))^2.
$$
Note that  
there exists at least one minimizer because 
the parameter set corresponding to $\calF$ is a compact set and $\eta$ is a continuous function.
$\fhat$ needs not necessarily be the exact minimizer but it could be an approximated minimizer.
We, however,  assume $\fhat$ is the exact minimizer for theoretical simplicity. 
In practice, the empirical risk minimizer is obtained 
by using the back-propagation technique.
The regularization for the norm of the weight matrices and the bias terms are implemented by using the $L_2$-regularization and the drop-out techniques.

The generalization error of the empirical risk minimizer is bounded as in the following theorem.

\begin{Theorem}

For any $\delta > 0$ and $\lambda_\ell >0$, suppose that 
\begin{align}
m_\ell \geq 5 N_\ell(\lambda_\ell) \log\left(32 N_\ell(\lambda_\ell)/\delta \right)~~~
(\ell = 2,\dots,L).
\label{eq:mellConditionERM}
\end{align}
Then, there exists universal constants $C_1$ such that, for any $r > 0$ and $\tilde{r} > 1$, 
\begin{align*}
\|\fhat - \ftrue\|_{\LPi}^2 \leq 
& 
C_3
\Bigg\{
\tilde{r} \deltanone^2  + (\sigma^2 + \Rhatinf^2) \deltantwo^2
+
\frac{ (\Rhatinf^2 + \sigma^2)}{n}\left[
\log_+ \left(\frac{\sqrt{n}}{\min\{\sigma/\Rhatinf,1\}}\right)  + r
 \right] 
 \Bigg\}
\end{align*}
with probability $1-  \exp\left(-  \frac{n \deltanone^2{(\tilde{r}-1)^2}}{ 11 \Rhatinf^2}   \right) - 2\exp(- r)$
for every $r >0$ and $\tilde{r}' > 1$.
\end{Theorem}

The proof is given in Appendix \ref{sec:ERMproof}. 
This theorem can be shown by 
evaluating the covering number of the model $\calF$
and applying the local Rademacher complexity technique \citep{IEEEIT:Mendelson:2002,LocalRademacher,Koltchinskii,gine2006concentration}.

It is easily checked that the third term of the right side ($\frac{ (\Rhatinf^2 + \sigma^2)}{n}\big[
\log_+ \left(\frac{\sqrt{n}}{\min\{\sigma/\Rhatinf,1\}}\right)  + r \big]$) is smaller than 
the first two terms, therefore the generalization error bound can be simply evaluated as  
$$
\|\fhat - \ftrue\|_{\Ltwo}^2 = O_p(\deltanone^2 + \deltantwo^2).
$$
Based on a rough evaluation 
$$\deltanone^2 \simeq L \sum_{\ell=1}^L \lambda_\ell,~~~\deltantwo^2   \simeq \sum_{\ell=1}^L \frac{m_\ell m_{\ell + 1}}{n} \log(n),$$
and the constraint $m_\ell \gtrsim N_\ell(\lambda_\ell) \log(N_\ell(\lambda_\ell))$,
we can observe the bias-variance trade-off for the generalization error bound 
because, as $\lambda_\ell$ decreases, the required width of the internal layer $\mell$
increases by the condition \eqref{eq:mellConditionERM} and thus 
the deviation $\deltantwo$ in the finite dimensional model should increase.
In other words, if we want to construct a finite dimensional model which well approximates the 
true function, then a more complicated model is required and 
we should pay larger variance of the estimator.
A key notion for the bias-variance trade-off is the degree of freedom $N_\ell(\lambda_\ell)$
which expresses the ``complexity'' of the RKHS $\calH_\ell$ in each layer.
The degree of freedom of a complicated RKHS grows up faster than a simpler one
as $\lambda$ goes to 0.
This is also informative in practice because, 
to determine the width of each layer, the degree of freedom gives a good guidance.
That is, if the degree of freedom is small compared with the sample size, 
then we may increase the width of the layer.
An estimate of the degree of freedom can be computed from the trained network
by computing the Gram matrix corresponding to the kernel induced from the trained network
(where the kernel is defined by the finite sum instead of the integral form)
and using the eigenvalue of the Gram matrix.

To obtain the best generalization error bound, $(\lambda_\ell)_{\ell=1}^L$ should be tuned to balance the bias-variance terms
(and accordingly $(\mell)_{\ell=2}^L$ should also be fine-tuned).
The examples of the best achievable generalization error will be shown in Section \ref{sec:GenErrorExamples}.

\subsection{Bayes estimator}
\label{sec:GenErrorBound}

In this section, we formulate a Bayes estimator and derive its generalization error.
To define the Bayes estimator, we just need to specify the prior distribution.
Let $\calB_d(C)$ be the ball in the Euclidean space $\Real^d$ with radius $C > 0$ 
$(\calB_d(C) = \{x \in \Real^d \mid \|x\| \leq C\})$,
and $\Unif(\calB_d(C) )$ be the uniform distribution on the ball $\calB_d(C)$.
Since Theorem \ref{th:FiniteApprox} ensures the norms of $\Well{\ell}$ and $\bell{\ell}$
are bounded above by $\Rbar$ and $\Rbarb$,
it is natural to employ a prior distribution 
that possesses its support on the set of parameters with norms not greater than those norm bounds. 
Based on this observation, we employ uniform distributions on balls with the radii indicated above 
as a prior distribution:
$$
\Well{\ell} \sim \Unif (  \calB_{m_{\ell + 1} \times m_{\ell}}(\Rbar)),~~
\bell{\ell} \sim \Unif (\calB_{m_{\ell + 1}}(\Rbarb)).
$$
In practice, the Gaussian distribution is also employed
instead of the uniform distribution.
However, the Gaussian distribution 
does not give good tail probability bound for the infinity norm of the deep neural network model.
That is crucial to develop the generalization error bound.
For this reason, we decided to analyze the uniform prior distribution. 

The prior distribution on the parameters $(\Well{\ell},\bell{\ell})_{\ell=1}^L$ induces the distribution of the function $f$
in the space of continuous functions endowed with the Borel algebra corresponding to the $L_\infty(\Real^{\dx})$-norm.
We denote by $\Pi$ the induced distribution.
Using the prior, the posterior distribution is defined via the Bayes principle:
$$
\Pi(\dd f | D_n) = \frac{\exp(- \sum_{i=1}^n \frac{(y_i - f (x_i))^2}{2\sigma^2}) \Pi(\dd f)}{\int \exp(- \sum_{i=1}^n \frac{(y_i - f' (x_i))^2}{2\sigma^2}) \Pi(\dd f')}.
$$
Since the purpose of this paper is to give a theoretical analysis for the generalization error, we do not pursue the computational issue of 
the Bayesian deep learning. See, for example, \citet{hernandez2015probabilistic,blundell2015weight} for practical algorithms.

The following theorem gives 
how fast the Bayes posterior contracts around the true function.


\begin{Theorem}
\label{th:PosteriorContraction}

Fix arbitrary $\delta > 0$ and $\lambda_\ell >0~(\ell=1,\dots,L)$, 
and suppose that the condition \eqref{eq:mellConditionERM} on $\mell$ is satisfied.
Then, for all $r \geq 1$, 
the posterior tail probability can be bounded as 
\begin{align*}
& \EE_{D_n}\left[  \Pi(f : \|f - \ftrue \|_{\Ltwo(P(X))} \geq (\deltanone + \sigma \deltantwo) r \sqrt{\max\{ 12, 33 \textstyle \frac{\Rhatinf^2}{\sigma^2} \}}   | D_n) \right] \\
&\leq  \exp\left[-n\deltanone^2 \frac{ (r^2-1)^2}{11 \Rhatinf^2}\right] 
+ 12 \exp\left(- n (\deltanone +  \sigma \deltantwo)^2 \frac{ r^2}{8\sigma^2} \right).
\end{align*}

\end{Theorem}

The proof is given in Appendix \ref{sec:PosteriorContractionProof}.
The proof is accomplished by using the technique for non-parametric Bayes methods \citep{AS:Ghosal+Ghosh+Vaart:2000,AS:Vaart&Zanten:2008,JMLR:Vaart&Zanten:2011}.
Roughly speaking this theorem indicates that 
the posterior distribution concentrates 
in the distance $\deltanone + \sigma \deltantwo$ from the true function $\ftrue$.
The tail probability is sub-Gaussian
and thus the posterior mass outside the distance $\deltanone + \sigma \deltantwo$ from the true function 
rapidly decrease.
Here we again observe that there appears bias-variance trade-off between $\deltanone$ and $\deltantwo$.
This can be understood essentially in the same way as the empirical risk minimization.

From the posterior contraction rate, we can derive the generalization error bound of the 
posterior mean.
\begin{Corollary}
Under the same setting as in Theorem \ref{th:PosteriorContraction},
there exists a universal constant $C_1$ such that 
the generalization error of the posterior mean $\fhat$ is bounded as 
\begin{align*}
& \EE_{D_n}\left[   \|\fhat - \ftrue \|_{\Ltwo(P(X))}^2 \right]  
\leq C_1 \max\left\{12,{\textstyle \frac{33\Rhatinf^2}{\sigma^2}}\right\} 
\left[ \left(1 + \frac{\Rhatinf}{\sqrt{n \deltanone^2}}  \right) (\deltanone^2 + \sigma^2 \deltantwo^2)+ \frac{\sigma^2}{n}  \right].
\end{align*}
\end{Corollary}

Therefore, for sufficiently large $n$ such that $n \deltanone^2/\Rhatinf^2 \geq 1$ (which is the regime of our interest),
the generalization error is simply bounded as 
$$
\|\fhat - \ftrue \|_{\Ltwo(P(X))}^2 =O_p\left(
{\textstyle \max\{1,{\textstyle \frac{\Rhatinf^2}{\sigma^2}}\}}  
 (\deltanone^2 + \sigma^2 \deltantwo^2) \right).
$$

\subsection{Examples}

\label{sec:GenErrorExamples}

Here, we give some examples of the generalization error bound.
We have seen that both of the empirical risk minimizer and the Bayes estimators have a simplified generalization error bound as 
$$
\|\fhat - \ftrue \|_{\Ltwo(P(X))}^2 = O_p(\deltanone^2 + \deltantwo^2) = O_p\left( L \sum_{\ell=2}^L \Rbar^{L-\ell + 1}
\lambda_\ell + \sum_{\ell=1}^L \frac{m_\ell m_{\ell+1}}{n} \log(n) \right),
$$
by supposing $\sigma$, $\Rhatinf$ 
and $\sqrt{\conedelta^L}R^{L}$ are in constant order. 
We evaluate the bound under the best choice of $\mell$ balancing the bias-variance trade-off.  

One way to balance the terms is to set $\lambda_\ell$ so that 
$$
\sum_{\ell=2}^L \lambda_\ell =  \sum_{\ell=1}^L \frac{m_\ell m_{\ell+1}}{n}
$$
where the $\log(n)$-factor and $L$ are dropped for simplicity.
Since the inequality of arithmetic sum geometric mean gives $\sum_{\ell=1}^L \frac{m_\ell m_{\ell+1}}{n} \leq \sum_{\ell=1}^{L+1} \frac{m_\ell^2}{n}$,
we may set $m_\ell$ to satisfy
\begin{align}
\lambda_\ell = \frac{m_\ell^2}{n}~~~(\ell=2,\dots,L). 
\label{eq:lambdaellmellbound}
\end{align}
Considering this relation and the constraint $m_\ell \gtrsim N_{\ell}(\lambda_\ell) \log(N_{\ell})$ (\Eqref{eq:mellConditionERM}),
we can estimate the best width $m_\ell$ that minimizes the upper bound of the generalization error. 


\subsubsection{Finite dimensional internal layer}
\label{sec:FiniteDimExample}

If all RKHSs are finite dimensional, say $m^*_\ell$-dimensional.
Then $N_\ell(\lambda) \leq m^*_{\ell}$ for all $\lambda \geq 0$.
Therefore, by setting $\lambda_\ell = 0~(\forall \ell)$, 
the generalization error bound is obtained as 
\begin{align}
\label{eq:deepNNconvergenceFiniteDim}
\|\fhat - \ftrue\|_{\LPi}^2 \lesssim \frac{\sigma^2 + \Rhatinf^2}{n} \sum_{\ell=1}^L m^*_{\ell}m^*_{\ell+1} \log(n),
\end{align}
where we omitted the factors depending only on $\log(\Rbar \Rbarb \hat{G})$.
Note that, although there appears the $L_\infty$-norm bound $\Rhatinf$,
this convergence rate is 
independent of the Lipschitz constant $\Rbar^{L-\ell+1}$ and $\Rbarb$ 
up to $\log$-order
but is solely dependent on the number of parameters.
Moreover, the convergence rate is $O(\log(n)/n)$ in terms of the sample size $n$.
This is much faster than the existing bounds that utilize the Rademacher complexity because their bounds are $O(1/\sqrt{n})$.
This result matches more precise arguments for a finite dimensional 3-layer neural network based on asymptotic expansions \citep{fukumizu1999generalization,watanabe2001learning}
which also showed the generalization error of the 3-layer neural network can be evaluated as $O((m_1^*m_2^* + m_2^* m_3^*)/n)$.

\subsubsection{Polynomial decreasing rate of eigenvalues}

We assume that the eigenvalue $\mu_j^{(\ell)}$ decays in polynomial order 
as 
\begin{equation}
\label{eq:mujboundcondition}
\mu_j^{(\ell)} \leq a_\ell j^{-\frac{1}{s_\ell}},
\end{equation}
for a positive real $0 < s_\ell <1$ and $a_\ell >0$.
This is a standard assumption in the analysis of kernel methods \citep{FCM:Caponetto+Vito:2007,Book:Steinwart:2008},
and it is known that this assumption is equivalent to the usual covering number assumption \citep{COLT:Steinwart+etal:2009}.
For small $s_\ell$, the decay rate is fast and it is easy to approximate the kernel 
by another one corresponding to a finite dimensional subspace. 
Therefore small $s_\ell$ corresponds to a simple model and large $s_\ell$ corresponds to a complicated model.
In this setting, the degree of freedom is evaluated as 
\begin{align}
N_\ell(\lambda_\ell) 
\lesssim \left( \lambda_\ell/a_\ell \right)^{-s_\ell}. 
\label{eq:Nlambdabound_eigenvalue}
\end{align}
This can be shown as follows:
for any positive integer $M$, the degree of freedom can be bounded as 
\begin{align*}
N_\ell(\lambda_\ell) & = \sum_{j=1}^\infty \frac{\mu_j^{(\ell)}}{\mu_j^{(\ell)} + \lambda_\ell}
\leq  M + \sum_{j=M+1}^\infty \frac{\mu_j^{(\ell)}}{\lambda_\ell} 
\leq M + \frac{a_\ell}{\lambda_\ell} \int_{M}^\infty x^{-1/s_\ell}  \dd x  \\
& \leq M + (a_\ell/\lambda_\ell) (1-1/s_\ell)^{-1} M^{1-1/s_\ell}.
\end{align*}
Letting $M = \left\lceil (a_\ell/\lambda_\ell) (1-1/s_\ell )^{-1} \right\rceil^{s_\ell}$ to balance the first and the second term, we obatain 
the evaluation \eqref{eq:Nlambdabound_eigenvalue}.
Hence, we can show that, according to \Eqref{eq:lambdaellmellbound},
$$
\lambda_\ell =  a_\ell^{\frac{2s_\ell}{1 + 2s_\ell}} n^{- \frac{1}{1+2 s_\ell}}
$$
gives the optimal rate, and we obtain the generalization error bound as
\begin{align}
\label{eq:deepNNconvergencePoly}
\|\fhat - \ftrue\|_{\LPi}^2 \lesssim
  L \sum_{\ell=2}^L (\Rbar \vee 1)^{2(L-\ell+1)}  a_\ell^{\frac{2s_\ell}{1+2s_\ell}} n^{-\frac{1}{1+2 s_\ell}} \log \left(n\right) + \frac{\dx^2}{n} \log(n),
\end{align}
where we omitted the factors depending on $s_\ell, \log(\Rbar \Rbarb \hat{G})$,
$\sigma^2$ and $\Rhatinf$.
This indicates that the complexity $s_\ell$ of the RKHS affects the convergence rate directly.
As expected, if the RKHSs are simple (that is, $(s_\ell)_{\ell=2}^L$ are small), 
we obtain faster convergence.

\subsubsection{One internal layer: kernel method}

Finally, we consider a simple but important situation 
in which there is only one internal layer ($L=2$).
In this setting, we only need to adjust $m_2$ because the dimensions of input and output are fixed as $m_1 = \dx$ and $m_3 =1$.
We assume that the same condition \eqref{eq:mujboundcondition} for $\ell=2$.
Then, applying the condition \eqref{eq:lambdaellmellbound}, 
\begin{align*}
& \lambda_2 =  \frac{(\dx +1)m_2}{n} ~~~
\Rightarrow ~~~ 
\lambda_2 \simeq  a_2^{\frac{s_2}{1+s_2}} \left(\frac{n}{\dx + 1}\right)^{-\frac{1}{1+s_2}} \log(n)
\end{align*}
gives the optimal convergence rate. Actually, we obtain
$$
\|\fhat - \ftrue\|_{\LPi}^2 \lesssim ((\Rbar\vee 1)^2 a_2^{\frac{s_2}{1+s_2}})  (\dx+1)^{\frac{1}{1+s_2}}  n^{-\frac{1}{1+s_2}} \log(n).
$$
This convergence rate is equivalent to the minimax optimal convergence rate 
of the kernel ridge regression \citep{FCM:Caponetto+Vito:2007,COLT:Steinwart+etal:2009} (up to
constant and $\log(n)$ factors).
It is known that the kernel method corresponds to the 3-layer neural network 
with an infinite dimensional internal layer.
In that sense, our analysis includes that of kernel methods.
In particular, the finite dimensional approximation we performed in Section \ref{sec:FiniteApproximation}
can be seen as the kernel quadrature rule \citep{bach2015equivalence}.
Thus, the analysis here ensures that the kernel quadrature rule can achieve the optimal rate 
as a byproduct of the neural network analysis.
In that sense, we can say that the deep neural network is a method that constructs an optimal kernel 
in a layer-wise manner
using a kernel quadrature rule.

Some concrete examples have bee investigated in \cite{JMLR:v18:14-546} for the three layer neural network.
However, the analysis does not assume that the loss function is strongly convex,
and thus the local Rademacher complexity analysis is not applied.
Consequently, the convergence rate is slower than $O(1/\sqrt{n})$.

The convergence rate is $O(n^{-\frac{1}{1+s_2}})$ when $L=2$, but 
we have already observed that the convergence rate \eqref{eq:deepNNconvergencePoly} of the deep neural network 
is basically $\sum_{\ell=2}^L n^{-\frac{1}{1+2 s_\ell}}$:
The sample complexity in each layer is slow for deep neural network (there is a factor $2$ before $s_\ell$).
This is because we need to estimate a matrix in each layer for deep neural network
and the number of output grows up as the sample size increases, and as a result,  
the number of parameters that should be estimated is much larger 
than the 3-layer neural network in which the dimensions of the input and output are fixed.


\section{Relations to existing work}

In this section, we describe the relation of our work to existing works

The sample complexity of deep neural network has been extensively studied 
by analyzing its Rademacher complexity.
For example, \citet{bartlett1998sample} characterized the generalization error of 
a 3-layer neural network by the norm of the weight vectors instead of the number of parameters.
\citet{koltchinskii2002empirical} studied more general deep neural network and 
derived its generalization error bound of deep neural network under a norm constraint.
They showed the Rademacher complexity of the deep neural network is bounded by 
the sum of those of single layer neural networks. 
This is similar to our generalization error bound obtained in  Eqs.\eqref{eq:deepNNconvergencePoly} and \eqref{eq:deepNNconvergenceFiniteDim}.
More recently, 
\citet{COLT:Neyshabur+Tomioka+Srebro:2015}
analyzed the Rademacher complexity of the deep neural network 
based on the norms of the weight matrix ($\{\Well{\ell}\}_{\ell=1}^L$ in our paper).
\citet{sun2015large} also derived the Rademacher complexity and showed the 
generalization error under a large margin assumption.
As consequences of these studies,
they derived 
the following type of inequalities:
$$
\EE_{X,Y}[l(\fhat(X),Y)] \leq \frac{1}{n}\sum_{i=1}^n l(\fhat(x_i),y_i) + \frac{C R^{2L} }{\sqrt{n}}
$$
with high probability 
where $l$ is a loss function, $\fhat$ is the empirical risk minimizer with or without regularization and $R$ is a norm bound of the internal layers
(the definition of the norm differs between papers).
Basically, these studies considered the finite dimensional situation 
which was studied in Section \ref{sec:FiniteDimExample} as a special case of our analysis,
and a connection to an infinite dimensional model has not been closely discussed.
In particular, the bias-variance trade-off has not been analyzed.
Moreover, 
the generalization error is $O(1/\sqrt{n})$ which is much slower dependency 
on the sample size than 
that of our rate $O(1/n)$.
This is a big difference from the existing analysis.
To improve the convergence rate to faster one for empirical risk minimization approaches, so called  {\it local Rademacher complexity} 
was important in our analysis. 
Moreover, we have observed that the Bayesian analysis also gave faster convergence rate.

Analysis of the bias-variance trade-off in three layer neural network from the kernel point of view has been investigated by \cite{JMLR:v18:14-546}.
The analysis is given for several concrete examples. However, the loss function is not assumed to be strongly convex, and thus the obtained rate 
is not faster than $O(1/\sqrt{n})$.

Another important topic for the analysis of the generalization ability 
is VC-dimension analysis.
VC-dimension of the deep neural network has been studied by, for example, \citet{bartlett1998almost,karpinski1997polynomial,goldberg1995bounding}.
However, VC-dimension is a notion independent of the input distributions.
On the other hand, the degree of freedom considered in our paper depends on the input distribution and is more data specific. 
Hence, our analysis gives tighter bound and could be practically more useful.

In our analysis, the kernel formulation of deep neural network model was the key ingredient for the analysis.
Some authors have proposed methods that utilize the representation in the internal layers as a feature map 
into some RKHS as in our formulation.
For example, \citet{cho2009kernel,mairal2014convolutional,NIPS2016_6184} have proposed novel methods 
to construct a kernel via deep learning.
Their purpose is to suggest a new hierarchical method to construct a kernel, and their proposed methods are different from the ordinary deep learning model.
On the other hand, our theoretical approach states that deep learning {\it itself} can be interpreted as a kind of kernel learning.
Moreover, their studies are not for theories but for methodologies.
Hence, our analysis and their studies are in complementary relationship, and our analysis would give theoretical support for their methods.

\section{Conclusion and future work}

In this paper, we proposed to use the integral form of deep neural network 
for generalization error analysis,
and based on that, we derived the generalization error bound of the empirical risk minimizer and the Bayes estimator.
The integral form enabled us to define an RKHS in each layer,
and import the theoretical techniques developed in kernel methods into the analysis of deep learning.
In particular, we defined the degree of freedom of each RKHS and 
showed that the approximation error between a finite dimensional model and the integral form can be characterized by the degree of freedom.
In addition to the approximation error, we also derived the estimation error in the finite dimensional model.
We have observed that there appears bias-variance trade-off 
depending on the size of the finite dimensional model.

Based on the analysis, we also derived generalization error bounds of some examples including 
the situation where the eigenvalues of the kernel function decay in a polynomial order, 
the one where the true model is finite dimensional,
and the one where there is only one internal layer.
We have observed that the analysis of the 3-layer neural network reproduces the optimal learning rate of 
the kernel method up to $\log(n)$-order.

Our theoretical frame-work offered a clear description of the bias-variance trade-off for deep learning.
This was particularly useful to determine the optimal widths of the internal layers.
We believe this study opens up a new direction of a series of theoretical analyses of deep learning.

There remain several topics to be studied.
One is characterization of the space of the deep neural network described by the integral form.
The integral form can approximate arbitrary function, but under the norm constraints as we have assumed,
it is unclear how large the function class is.
Solving this issue is interesting future work.

\section*{Acknowledgment}

This work was partially supported by MEXT kakenhi (25730013, 25120012,
26280009, 15H01678 and 15H05707), JST-PRESTO and JST-CREST (JPMJCR1304 and JPMJCR14D7).

\appendix

\section{Proof of Theorem \ref{th:FiniteApprox}}

\label{app:FiniteApprox}

\subsection{Approximation error bound for the finite dimensional model}


To derive the approximation error bound,
we utilize the following proposition that was proven by 
\citet{bach2015equivalence}.

\begin{Proposition}
\label{prop:BachFiniteApprox}

For $\lambda > 0$, there exists a probability density $q_\ell(\tau)$ with respect to the measure $Q_\ell$ such that,
for any $\delta \in (0,1)$,
i.i.d. sample $v_1,\dots,v_m$ from $q_\ell$ satisfies that
$$
\sup_{\|f\|_{\calH_\ell} \leq 1} \inf_{\beta \in \Real^m : \|\beta\|_2^2 \leq \frac{4}{m}} 
\left\| f - \sum_{j=1}^m \beta_j q_\ell(v_j)^{-1/2} \eta (F_{\ell -1}(\cdot, v_j)) \right\|_{\Ltwo(P(X))}^2 \leq 4 \lambda,
$$
with probability $1-\delta$, if 
$$
m \geq 5 N_\ell(\lambda) \log(16 N_\ell(\lambda)/\delta).
$$
\end{Proposition}

By the scale invaliance of $\eta$, $\eta(a x) = a \eta(x)~(a >0)$, we have the following proposition based on 
Proposition \ref{prop:BachFiniteApprox}.
\begin{Proposition}
\label{prop:ApproxFinite}
For $\lambda >0$, and any 
$1/2 > \delta >0$,
if  
$$
m \geq 5 N_\ell(\lambda) \log(16 N_\ell(\lambda)/\delta),
$$
then there exist $v_1,\dots v_m \in \calT_\ell$,  $w_1,\dots, w_m >0$ such that 
$$
\sup_{\|f\|_{\calH_\ell} \leq R}  \inf_{\beta \in \Real^m : \|\beta\|_2^2 \leq \frac{4R^2}{m}} 
\left\| f - \sum_{j=1}^m \beta_j \eta (w_j F_{\ell -1}(\cdot, v_j)) \right\|_{\Ltwo(P(X))}^2 \leq 4 \lambda R^2,
$$
and 
$$
\frac{1}{m} \sum_{j=1}^m w_j^2 \leq (1-2\delta)^{-1}.
$$

\end{Proposition}

\begin{proof}
Notice that $\EE[\frac{1}{m} \sum_{j=1}^m q_\ell (v_j)^{-1}] = \EE[q_\ell (v)^{-1}] = 
\int_{\calT_\ell} q_\ell (v)^{-1} q_\ell(v) \dd Q_\ell(v) = \int_{\calT_\ell} 1 \dd Q_\ell(v) =  1$,
thus an i.i.d. sequence $\{v_1,\dots,v_m\}$ satisfies $\frac{1}{m} \sum_{j=1}^m q_\ell (v_j)^{-1} \leq 1/(1-2 \delta)$ with probability $2\delta$
by the Markov's inequality.
Combining this with Proposition \ref{prop:BachFiniteApprox}, the i.i.d. sequence 
$\{v_1,\dots,v_m\}$ and $w_j = q_\ell (v_j)^{-1/2}$ satisfies the condition in the statement with probability 
$1-(\delta + 1- 2\delta ) = \delta >0$. 
This ensures the existence of sequences $\{v_j\}_{j=1}^m$ and $\{w_j\}_{j=1}^m$ that satisfy the assertion.
\end{proof}


From now on, we define 
$$
\czero = 4,~\cone = 4,~\cdelta = (1-2\delta)^{-1}.
$$
The next theorem gives the proof of the approximation error bound in Theorem \ref{th:FiniteApprox}. 
The $L_\infty$-norm bound of $\ftrue$ is given later in Lemma \ref{supplemma:SupnormBounds}.
Substituting $\delta \leftarrow \delta/2$ into the statement in the following 
Lemma \ref{th:ApproxErrorBound}
and letting $\conedelta = \cone c_{\delta/2}$, we obtain Theorem \ref{th:FiniteApprox}.

\begin{Lemma}[Approximation error bound of the nonparametric model]
\label{th:ApproxErrorBound}
For any $1/2 > \delta > 0$ and given $\lambda_\ell > 0$, let $m_\ell \geq 5 N_\ell(\lambda_\ell) \log(16 N_\ell(\lambda_\ell)/\delta)$. 
Then there exist
$
W^{(\ell)} \in \Real^{m_{\ell + 1} \times m_\ell}
$
and 
$
\bell{\ell} \in \Real^{m_{\ell +1}}
$
($\ell=1,\dots,L$)
where $m_{L+1} = 1$ and $m_1 = \dx$ such that
\begin{align*}
&\|W^{(\ell)}\|_{\F}^2 \leq \cone \cdelta   R^2,~~
\|\bell{\ell}\|_2 \leq \sqrt{\cdelta} \Rb~~(\ell = 1,\dots,L-1), \\
& \|W^{(L)}\|_{\F}^2 \leq \cone R^2,~~\|\bell{L}\|_2 \leq \Rb,
\end{align*}
and 
\begin{align*}
\|\ftrue - \fstar\|_{\Ltwo(P(X))}
\leq  
\sum_{\ell = 2}^L \sqrt{ (\cone \cdelta)^{L-\ell} \czero }  R^{L-\ell + 1} \sqrt{\lambda_\ell}.
\end{align*}

\end{Lemma}

\begin{proof}

We construct the asserted finite dimensional network recursively from $\ell = L$ to $\ell=1$.
Let $\{v_j^{(\ell)}\}_{j=1}^{\mell}$ and $\{w_j^{(\ell)}\}_{j=1}^{\mell}$ be the sequences given in Proposition \ref{prop:ApproxFinite}.
Let $\widehat{\calT}_\ell = \{v_j^{(\ell)}\}_{j=1}^{\mell}$.
With slight abuse of notation, we identify $\fstar_\ell: \Real^{\mell} \to \Real^{m_{\ell+1}}$ to a function
$\fstar_\ell:\widehat{\calT}_\ell \to \widehat{\calT}_{\ell +1}$ in a canonical way.  
For a function $F: \Real^{\dx} \times \widehat{\calT}_\ell \to \Real$, we denote by 
$\fstar_\ell[F](x,v_i^{(\ell + 1)})$ to express $\fstar_\ell[F(x,\cdot)](v_i^{(\ell + 1)}) = \sum_{j=1}^{m_{\ell}} \Well{\ell}_{i,j} F(x,v_j^{(\ell)}) + \bell{\ell}_i$ for
$v_i^{(\ell + 1)} \in \widehat{\calT}_{\ell + 1}$.
When we write $\fstar_\ell[F]$ for $F:\Real^{\dx}  \times \calT_\ell \to \Real~((x,v) \mapsto F(x,v))$,
we deal with $F$ as a restriction of $F$ on $\Real^{\dx} \times \widehat{\calT}_\ell$.
We define the output from the $\ell$-th layer of the approximated network $\fstar$ as
$F^*_{\ell}(x,v)$ for $v \in \widehat{\calT}_\ell$ and $x \in \Real^{\dx}$.
More precisely, it is recursively defined as $F^*_{\ell}(x,v) = \fstar_\ell[F^*_{\ell-1}](x,v)$.

We use an analogous notation for other networks such as $\ftrue_\ell$. 
That is, $\Ftrue_\ell(x,v) = (\ftrue_\ell \circ \dots \circ \ftrue_1(x))(v)$ for $v\in \calT_\ell$ and $x \in \Real^{\dx}$, 
and $\Ftrue_\ell(x,v) = \ftrue_\ell[\Ftrue_{\ell-1}](x,v)$.

\noindent {\bf Step 1} (the last layer, $\ell = L$).

We consider the following approximation of the $L$-th layer (the last layer):
Remember that $m_{L+1} =1$ and thus the output from the $L$-th layer is just one dimensional.
We denote by  $\calT_{L+1} = \{1\}$ which is the index set of the output (which is just a singleton consisting of an element 1).
As a candidate of a good approximation to the true $L$-th layer, define 
\begin{align}
\fstarsub_{L}[F_{L-1}](x,1) = \sum_{j=1}^{m_L} \sqrt{m_L} \beta^{(L)}_{j} \eta \left(\frac{1}{\sqrt{m_L}} w_j^{(L)} F_{L-1}(x, v_j^{(L)}) \right) + b_L
\label{eq:fstarLFiniteApproximate}
\end{align}
by $\beta^{(L)} \in \Real^{m_{L}}$ and $w^{(L)} \in \Real^{m_L}$ satisfying 
$\|\beta^{(L)}\|_{2}^2 \leq \frac{1}{m_L} \cone R^2$ and $\|w^{(L)}\|_2^2 \leq m_L \cdelta$.
Here, define that 
$$
W^{(L)}_{1,:} = \sqrt{m_L} {\beta^{(L)}}^\top,~~\bell{L} = (\btrue_L(1)). 
$$
Note that the model \eqref{eq:fstarLFiniteApproximate} can be rewritten as 
\begin{align}
\fstarsub_{L}[F_{L-1}](x,1) = \sum_{j=1}^{m_L} W^{(L)}_{1,j} \eta (\sqrt{m_L}^{-1} w_j^{(L)} F_{L-1}(x, v_j^{(L)})) + \bell{L}_1.
\notag 
\end{align}
Because of  Proposition \ref{prop:ApproxFinite} and Assumption \ref{ass:hbnormbounds}, 
the norms of the weight $\Well{L}$ and the bias $\bell{L}$ are bounded as 
\begin{align}
\|W^{(L)}\|_F = \|W^{(L)}_{1,:}\|_2 \leq \sqrt{\cone} R,~~ 
\|\bell{L}\|_2 = |b_L| \leq \Rb.
\label{eq:WbLbounds}
\end{align}

By the Cauchy-Schwartz inequality and the Lipschitz continuity of $\eta$, we have that 
\begin{align*}
& | \fstarsub_{L}[F_{L-1}](x,1) - \fstarsub_{L}[F'_{L-1}](x,1) | \\
& \leq 
|\sum_{j=1}^{m_L} W^{(L)}_{1,j} (\eta ( \sqrt{m_L}^{-1}w_j^{(L)} F_{L-1}(x, v_j^{(L)})) - \eta (\sqrt{m_L}^{-1} w_j^{(L)} F'_{L-1}(x, v_j^{(L)})))| \\
& \leq \|W^{(L)}_{1,:}\|_2 \sqrt{m_L}^{-1} \|(w^{(L)}_j (F_{L-1}(x, v_j^{(L)})   - F'_{L-1}(x, v_j^{(L)})))_{j=1}^{m_L}\|_{2} \\
& \leq \|W^{(L)}_{1,:}\|_2 \sqrt{m_L}^{-1} \|w^{(L)}\|_2 
\| (F_{L-1}(x, v_j^{(L)})   - F'_{L-1}(x, v_j^{(L)}))_{j=1}^{m_L}\|_{\max} \\
&
\leq
\sqrt{\cone R^2} \sqrt{\cdelta m_L/m_L} \| (F_{L-1}(x, v_j^{(L)})   - F'_{L-1}(x, v_j^{(L)}))_{j=1}^{m_L}\|_{\max} \\
& =
\sqrt{\cone \cdelta} R  \| (F_{L-1}(x,v_j^{(L)})   - F'_{L-1}(x,v_j^{(L)}))_{j=1}^{m_L}\|_{\max},
\end{align*}
for $F_{L-1},F'_{L-1}: \widehat{\calT}_{L} \times \Real^{\dx} \to \Real$.
Moreover, Proposition \ref{prop:ApproxFinite} ensures that
$\beta^{(L)}$ and $w^{(L)}$ can be taken so that 
\begin{align*}
\| \fstarsub_{L}[\Ftrue_{L-1}](\cdot,1) - \ftrue_{L}[\Ftrue_{L-1}](\cdot,1) \|_{\Ltwo(P(X))}^2
\leq \czero \lambda_L R^2.
\end{align*}
Hereinafter, we fix $\beta^{(L)}$ and $w^{(L)}$ so that this inequality and the norm bound \eqref{eq:WbLbounds} are satisfied.

\noindent {\bf Step 2} (internal layers for $\ell = 2,\dots,L-1$).
As for the $\ell$-th internal layer, 
we consider the following approximation:
$$
\fstarsub_\ell[g](v_i^{(\ell+1)}) = \sum_{j=1}^{\mell} \sqrt{m_\ell} \beta^{(\ell)}_{i,j} \eta (\sqrt{m_\ell}^{-1} w_j^{(\ell)} g(v_j^{(\ell)})) + \btrue_\ell(v_i^{(\ell+1)}),
$$
for $g: \widehat{\calT}_{\ell} \to \Real$
with $\beta^{(\ell)} \in \Real^{m_{\ell+1} \times \mell}$ and $w^{(\ell)} \in \Real^{\mell}$ 
satisfying 
$\|\beta^{(\ell)}_{j,:}\|_{2}^2 \leq \frac{1}{m_\ell}  \cone R^2~(\forall j =1,\dots,m_{\ell+1})$ 
and $\|w^{(\ell)}\|_2^2 \leq m_\ell \cdelta$.
Then, the Lipschitz continuity of $\fstarsub_\ell$ can be shown as 
\begin{align*}
& | \fstarsub_{\ell}[F_{\ell-1}](x,v_i^{(\ell + 1)}) - \fstarsub_{\ell}[F'_{\ell-1}](x,v_i^{(\ell + 1)}) | \\
& \leq 
\left| 
\sum_{j=1}^{m_{\ell}} \sqrt{m_\ell} \beta^{(\ell)}_{i,j} (\eta (\sqrt{m_\ell}^{-1} w_j^{(\ell)} F_{\ell-1}(x, v_j^{(\ell)})) - \eta (\sqrt{m_\ell}^{-1} w_j^{(\ell)} F'_{\ell-1}(x, v_j^{(L)}))) 
\right| \\
& \leq \|\beta^{(\ell)}_{i,:}\|_2 \|w^{(\ell)}\|_2 
\| (F_{\ell-1}(x, v_j^{(\ell)})   - F'_{\ell-1}(x, v_j^{(\ell)}))_{j=1}^{m_\ell}\|_{\max} \\
&
\leq
\sqrt{\frac{\cone}{m_\ell}} R \sqrt{\cdelta m_\ell} \| (F_{\ell-1}(x, v_j^{(\ell)})   - F'_{\ell-1}(x, v_j^{(\ell)}))_{j=1}^{m_\ell}\|_{\max} \\
& =
\sqrt{\cone \cdelta} R  \| (F_{\ell-1}(x, v_j^{(\ell)})   - F'_{\ell-1}(x, v_j^{(\ell)}))_{j=1}^{m_L}\|_{\max},
\end{align*}
for any $v_i^{(\ell + 1)} \in \widehat{\calT}_{(\ell + 1)}$.
Proposition \ref{prop:ApproxFinite} asserts that 
there exit $\beta^{(\ell)}$ and $w^{(\ell)}$
that give an upper bound of the approximation error of the $\ell$-th layer as
\begin{align*}
\max_{j=1,\dots,\mell} \| \fstarsub_{\ell}[\Ftrue_{\ell-1}](\cdot,v_j^{\ell+1})  - \ftrue_{\ell}[\Ftrue_{\ell-1}](\cdot,v_j^{\ell+1}) \|_{\Ltwo(P(X))}^2
\leq \czero \lambda_\ell R^2.
\end{align*}
Finally, let 
$$
W^{(\ell)}_{ij} = \sqrt{\frac{m_{\ell}}{m_{\ell+1}}} \beta^{(\ell)}_{ij} w_i^{(\ell + 1)},
~~
\bell{\ell} =  \frac{1}{\sqrt{m_{\ell + 1}}}(w_1^{(\ell + 1)} \btrue_\ell(v_1^{(\ell+1)}),\dots,w_{m_{\ell+1}}^{(\ell + 1)} \btrue_\ell(v_{m_{\ell + 1}}^{(\ell+1)}))^\top,
$$ 
then, by Assumption \ref{ass:hbnormbounds} and Proposition \ref{prop:ApproxFinite}, the norms of these quantities can be bounded as 
\begin{align*}
\|W^{(\ell)}\|_{\F}^2 &  = \frac{m_{\ell}}{m_{\ell+1}} \sum_{i=1}^{m_{\ell + 1}} \sum_{j=1}^{m_{\ell}}\beta^{(\ell)^2}_{ij} w_i^{(\ell + 1) 2} \\
&  \leq \frac{m_{\ell}}{m_{\ell+1}} \sum_{i=1}^{m_{\ell + 1}} w_i^{(\ell + 1) 2}  \frac{\cone R^2}{\mell} 
\leq \cone \cdelta 
R^2,
\end{align*}
and 
\begin{align*}
\|\bell{\ell}\|_2^2 \leq \frac{1}{m_{\ell + 1}} \sum_{j=1}^{m_{\ell + 1}} {w^{(\ell +1)}}^2    \Rb^2 \leq \cdelta \Rb^2.
\end{align*}

\noindent {\bf Step 3} (the first layer, $\ell = 1$).

For the first layer, let 
$$
\fstarsub(x,v_i^{(2)}) = \sum_{j=1}^{\dx} \htrue_1(v_i^{(2)},j) Q_1(j) x_j + \btrue_1(v_i^{(2)})
$$
for $v_i^{(2)} \in \widehat{\calT}_2$.
By the definition of $\ftrue$, it holds that 
$$
\fstarsub(x,v_i^{(2)}) = \ftrue(x,v_i^{(2)}).
$$

Let $W^{(1)} =  \frac{1}{\sqrt{m_2}} (Q_1(j) w^{(2)}_i \htrue_1(v_i^{(2)},j))_{i,j} \in \Real^{m_2 \times \dx}$ and $
b^{(1)} = \frac{1}{\sqrt{m_2}}  (w_1^{(2)} \btrue_1(1),\dots,w_{m_2}^{(2)} \btrue_1(m_2))^\top \in \Real^{m_2}$. 
Then, by Assumption \ref{ass:hbnormbounds} and Proposition \ref{prop:ApproxFinite}, 
it holds that 
\begin{align*}
\|W^{(1)} \|_{\F}^2= & 
\sum_{i=1}^{m_2} \sum_{j=1}^{\dx}  \frac{1}{m_2} {w_i^{(2)}}^2 \htrue_1(v_i^{(2)},j)^2 Q_1(j)^{2}   \\
\leq & 
\left(\sum_{i=1}^{m_2} \frac{1}{m_2} {w_i^{(2)}}^2 \right)
\max_{1\leq i \leq m_2}
\left(\sum_{j=1}^{\dx}  \htrue_1(v_i^{(2)},j)^2 Q_1(j)^{2} \right)  \\
\leq & 
\cdelta
\max_{1\leq i \leq m_2}
\left(\sum_{j=1}^{\dx}  \htrue_1(v_i^{(2)},j)^2 Q_1(j) \right) 
\leq  
\cdelta R^2,
%
\end{align*}
and 
\begin{align*}
\|\bell{1}\|_2^2 \leq \frac{1}{m_1}\sum_{i=1}^{m_2} {w_i^{(2)}}^2 \Rb^2 \leq \cdelta \Rb^2.
\end{align*}

%

\noindent {\bf Step 4}.

Finally, we combine the results we have obtained above. 
Note that
\begin{align*}
& \| \ftrue_L \circ \ftrue_{L-1} \circ \cdots \circ \ftrue_{1} 
- \fstarsub_L \circ \fstarsub_{L-1} \circ \cdots \circ \fstarsub_{1} \|_{\Ltwo(P(X))} \\
=
& \| \ftrue_L \circ \ftrue_{L-1} \circ \cdots \circ \ftrue_{1} 
- \fstarsub_L \circ \ftrue_{L-1} \circ \cdots \circ \ftrue_{1}   \\ 
& 
~~~~~~~~~~~~~~~~~\vdots \\
&
+ \fstarsub_L \circ \dots \circ
\fstarsub_{\ell+1} \circ \ftrue_{\ell} \circ \ftrue_{\ell-1} \circ \dots \circ \ftrue_1 
-
\fstarsub_L \circ \dots \circ
\fstarsub_{\ell+1} \circ \fstarsub_{\ell} \circ \ftrue_{\ell-1} \circ \dots \circ \ftrue_1 
\\
& 
~~~~~~~~~~~~~~~~~\vdots \\
&
+ \fstarsub_L \circ  \cdots \fstarsub_{2}   \circ \ftrue_{1} 
-
\fstarsub_L \circ  \cdots \fstarsub_{2}   \circ \fstarsub_{1} 
\|_{\Ltwo(P(X))} \\
\leq & \sum_{\ell=1}^L
\|
\fstarsub_L \circ \dots \circ
\fstarsub_{\ell+1} \circ \ftrue_{\ell} \circ \ftrue_{\ell-1} \circ \dots \circ \ftrue_1 
-
\fstarsub_L \circ \dots \circ
\fstarsub_{\ell+1} \circ \fstarsub_{\ell} \circ \ftrue_{\ell-1} \circ \dots \circ \ftrue_1 
\|_{\Ltwo(P(X))}.
\end{align*}
Then combining the argument given above, we have
\begin{align*}
&
\|
\fstarsub_L \circ \dots \circ
\fstarsub_{\ell+1} \circ \ftrue_{\ell} \circ \ftrue_{\ell-1} \circ \dots \circ \ftrue_1 
-
\fstarsub_L \circ \dots \circ
\fstarsub_{\ell+1} \circ \fstarsub_{\ell} \circ \ftrue_{\ell-1} \circ \dots \circ \ftrue_1 
\|_{\Ltwo(P(X))} \\
\leq
&
(\sqrt{\cone \cdelta} R)^{L-\ell } (\sqrt{\czero \lambda_\ell} R) 
= \sqrt{ (\cone \cdelta)^{L-\ell} \czero } R^{L-\ell + 1}    \sqrt{\lambda_\ell}, 
\end{align*}
for $\ell = 2,\dots,L$. And the right hand side is 0 for $\ell=1$.
This yields that 
\begin{align*}
\|\ftrue - \fstarsub\|_{\Ltwo(P(X))}
\leq  
\sum_{\ell = 2}^L R^{L-\ell  + 1} \sqrt{ (\cone \cdelta)^{L-\ell} \czero  }  \sqrt{\lambda_\ell}.
\end{align*}
By substituting $W^{(\ell)}$ and $\bell{\ell}$ for $\ell = 1,\dots,L$ defined above into the 
definition of $\fstar$, then it is easy to see that 
$$
\fstar = \fstarsub
$$
as a function. Then, we obtain the assertion.

\end{proof}

\subsection{Bounding the $L_\infty$-norm}

\label{app:LinftyNormBound}

The next lemma shows the $L_\infty$-norm of  the true function $\ftrue$ and that of 
$f \in \calF$.
This gives 
the $L_\infty$-norm bound of every $f\in \calF$ in Lemma \ref{eq:fbounded}
and thus that of $\fstar$ in Theorem \ref{th:FiniteApprox}  because $\fstar \in \calF$.

\begin{Lemma}
\label{supplemma:SupnormBounds}
Under Assumptions \ref{ass:hbnormbounds}, \ref{ass:EtaCondition} and \ref{ass:xbounds}, the $L_\infty$-norms of $\ftrue$ and 
that of $f \in \calF$
are bounded as 
\begin{align*}
\|\ftrue\|_\infty 
& \leq R^{L} D_x + \sum_{\ell = 1}^L R^{L-\ell} \Rb,\\
\|f\|_\infty
& 
\leq (\sqrt{\cone \cdelta})^{L} R^{L} D_x + \sum_{\ell = 1}^L (\sqrt{\cone \cdelta} R)^{L-\ell} \Rbarb.
\end{align*}

\end{Lemma}

\begin{proof}

Suppose that 
$$
\|\Ftrue_{\ell -1}(x,\cdot) \|_{\Ltwo(Q_{\ell})} \leq G.
$$
Then, $\Ftrue_{\ell}$ can be bounded inductively:  
for all $\tau \in \calT_{\ell + 1}$
\begin{align*}
|\Ftrue_{\ell}(x,\tau) |
& 
= \left| \int_{\calT_\ell} \htrueell{\ell}(\tau,w)
\eta( \Ftrue_{\ell-1}(x,w)) \dd Q_\ell(w) + \btrue_\ell(\tau) \right| \\
&
\leq \| \htrueell{\ell}(\tau,\cdot) \|_{\Ltwo(Q_\ell)}
\| \Ftrue_{\ell-1}(x,\cdot) \|_{\Ltwo(Q_\ell)} + |\btrue_\ell(\tau)| \\
&
\leq  R G + \Rb, 
\end{align*}
by Assumption \ref{ass:hbnormbounds}.
Similarly, as for $\ell = 1$, it holds that,
for all $\tau \in \calT_2$ and $x \in \Real^{\dx}$,
\begin{align*}
|\ftrue_1(x,\tau)| &
= |\sum_{i=1}^{\dx} \htrueell{1}(\tau,i) x_i Q_1(i) + \btrue_{1}(\tau)| \\
& \leq |\sum_{i=1}^{\dx} \htrueell{1}(\tau,i) x_i Q_1(i)| + |\btrue_{1}(\tau)| \\
& \leq \|\htrueell{1}(\tau,\cdot)\|_{\Ltwo(Q_1)} \|x\|_{\Ltwo(Q_1)} + \Rb \\
& \leq R D_x + \Rb.
\end{align*}
Applying the same argument recursively, we have 
$$
\|\ftrue\|_\infty \leq R^{L} D_x + \sum_{\ell = 1}^L R^{L-\ell} \Rb.
$$

We can bound the $L_\infty$-norm of any $f \in \calF$ through a similar argument.
Note that
$\Well{\ell}$ satisfies $\|\Well{\ell}\|_{\F} \leq \sqrt{\cone \cdelta } R$ for $\ell = 1,\dots,L-1$,
$\Well{L}$ satisfies $\|\Well{L}\|_{\F} \leq \sqrt{\cone} R$, and 
$\bell{\ell}$ satisfies $\|\bell{\ell}\|_2 \leq \sqrt{\cdelta} \Rb $ by its construction.
Therefore, though a similar argument to the bound for $\ftrue$, we have that
\begin{align*}
\|f\|_\infty 
& \leq \sqrt{\cone} R 
\left[ \prod_{\ell=2}^{L-1} \left(\sqrt{\cone \cdelta } R\right)\right]
\sqrt{\cdelta }R D_x \\
& +  
\left(
\sum_{\ell=1}^{L-2}
\sqrt{\cone} R \left[ \prod_{\ell'=\ell+1}^{L-1} \left(\sqrt{\cone \cdelta}  R\right)  \right]
\sqrt{\cdelta} \Rb 
+ \sqrt{\cone} R \sqrt{\cdelta} \Rb +  \sqrt{\cdelta} \Rb
\right)
\\
& 
\leq 
(\cone \cdelta)^{L/2} R^{L} D_x + \sum_{\ell = 1}^L (\sqrt{\cone \cdelta} R)^{L-\ell} \Rbarb.
\end{align*}
\end{proof}

\section{Bounding the posterior contraction rate}
\label{sec:PosteriorContractionProof}

In this section, we prove Theorem \ref{th:PosteriorContraction}.
The proof is divided into two parts: 
posterior contraction rate with respect to the in-sample error 
(i.e., the empirical $L_2$-norm $\|f\|_n = \sqrt{\sum_{i=1}^n f(x_i)^2 /n }$) and that with respect to the out-of-sample error 
(i.e., the population $L_2$-norm $\|f\|_{\LPi} = \sqrt{\int f(X)^2 \dd P(X)}$).

Here, let 
$$
\epsilonn = \deltanone + \sigma \deltantwo,~~~~
\tilepsilonn = \deltanone + \deltantwo.
$$

\subsection{In-sample error}
\label{sec:AppInsampleError}

Here we show the in-sample error bound.
Let  $X_{n} = (x_1,\dots,x_n)$, $Y_{n} = (y_1,\dots,y_n)$ and $D_n = (X_n,Y_n)$.
For given $X_n$, the probability distribution of $Y_{n}$ associated with a function $f$ (i.e., $y_i = f(x_i) + \epsilon_i$)
is denoted by $P_{n,f}$.
The expectation of a function $h$ of $Y_n$ with respect to $P_{n,f}$ is denoted by $P_{n,f}(h)$.
The density function of $P_{n,f}$ with respect to $Y_n$ is denoted by $p_{n,f}$.

For $\rtil \geq 1$, let $\calA_{\rtil}$ be the event such that 
\begin{align*}
& \int \frac{p_{n,f}(Y_n)}{p_{n,\ftrue}(Y_n)} \Pi(\dd f) \geq \exp(- n \tilepsilonn^2 \rtil^2/\sigma^2) 
\Pi(f : \|f - \fstar \|_{\infty} \leq \deltantwo \rtil ).
\end{align*}
The probability of this event is bounded by Lemma \ref{lemm:PriorMass}.



Using a test function $\phi_n$ defined later
(here, a test function is a measurable function of $D_n$ that takes its value in $[0,1]$), we decompose the expected posterior mass as 
\begin{align}
& \EE\left[ \Pi(\|f - \ftrue \|_n \geq \sqrt{2} \epsilonn r  | D_{n} )  \right] \notag \\
\leq
&
 \EE\left [ \phi_n \right] 
+ 
P(\calA_{\rtil}^c) \notag \\
&
+ 
\EE[(1-\phi_n) \boldone_{\calA_{\rtil}} \Pi(f \in \calF^c | D_{n}) ]
\notag \\
&
+ 
\EE[(1-\phi_n) \boldone_{\calA_{\rtil}} \Pi(f \in \calF: \|f-\ftrue \|_n^2 \geq 2 \epsilon r^{2}| D_{n}) ] \notag \\
=:&
A_n + B_n + C_n + D_n,
\label{eq:decompIntABCD}
\end{align}
for $\epsilonn > 0$ where the expectation is taken with respect to $D_n = (X_n,Y_n)$ distributed from the true distribution.
We give an upper bound of $A_n$, $B_n$, $C_n$ and $D_n$ in the following.


~\\

\noindent {\it \bf Step 1.}

For arbitrary $r' >0$, define $C_{r'} = \{ f \in \calF \mid r' \leq 
\sqrt{n} \|f- \ftrue\|_n /\sigma \}$. 
We construct a maximum cardinality set $\Theta_{r'} \subset C_{r'}$ such that 
each $f,f'\in \Theta_{r'}$ satisfies $\sqrt{n}\|f-f'\|_n /\sigma \geq r'/2$.
Here we denote by $D(\epsilon,\calF,\|\cdot\|)$ the $\epsilon$-packing number of  a normed space $\calF$ attached with a norm $\|\cdot\|$.
Then, the cardinality of $\Theta_{r'}$ is equal to $D(r'/2,C_{r'},\sqrt{n}\|\cdot\|_n/\sigma)$.
Then,
following Lemma 13 of \cite{JMLR:Vaart&Zanten:2011},
one can construct a test $\tilde{\phi}_{r'}$ such that 
\begin{align*}
& P_{n,\ftrue} \tilde{\phi}_{r'} \leq 9 D(r'/2,C_{r'},\sqrt{n}\|\cdot\|_n/\sigma) e^{-\frac{1}{8} {r'}^2} 
\leq 9 D(r'/2,\calF,\sqrt{n}\|\cdot\|_n/\sigma) e^{-\frac{1}{8} {r'}^2}, \\
& \sup_{f \in C_{r'}} P_{n,f}(1-\tilde{\phi}_{r'}) \leq e^{-\frac{1}{8} {r'}^2}, 
\end{align*}
for any $r' > 0$.

Substituting $\sqrt{2} \sqrt{n} \epsilonn r/\sigma$ into $r'$
and denoting $\phi_n = \tilde{\phi}_{\sqrt{2} \sqrt{n} \epsilonn r/\sigma}$,
we obtain 
\begin{align}
P_{n,\ftrue} \phi_n &
\leq 9 e^{- \frac{1}{4 \sigma^2 }n \epsilonn^2 r^{2} + \log(D(r'/2,\calF,\sqrt{n}\|\cdot\|_n/\sigma)) } 
\label{eq:testInnerBound}
\\
\sup_{f\in C_{2 \sqrt{2} \sqrt{n} \epsilonn r}} P_{n,f}(1-\phi_n) & 
\leq e^{-\frac{1}{4 \sigma^2} n \epsilonn^2 r^2  }.
\label{eq:testOuterBound}
\end{align}

Hence, we just need to evaluate the (log-)packing number $\log(D(r'/2 ,\calF,\sqrt{n}\|\cdot\|_n/\sigma))$ where 
$r'= \sqrt{2 n} \epsilonn r/\sigma$.
It is known that the packing number is bounded from above by the internal covering number\footnote{
The $\epsilon$-internal covering number of a (semi)-metric space $(T,d)$ is the minimum cardinality of 
a finite set such that every element in $T$ is in distance $\epsilon$ from the finite set with respect to the metric $d$.
We denote by $N(\epsilon,T,d)$ the $\epsilon$-internal covering number of $(T,d)$.
},
and the packing number of unit ball in $d$-dimensional Euclidean space and that of the covering number is bounded as 
$$
D(\epsilon, \calB_d(1), \|\cdot\|) \leq N(\epsilon, \calB_d(1), \|\cdot\|) \leq \left( \frac{4+\epsilon}{\epsilon} \right)^d.
$$
Based on this we evaluate the packing number of $\calF$.

Let $f,f'\in \calF$ be two functions corresponding to 
parameters $(W^{(\ell)},b^{(\ell)})_{\ell=1}^L$ and $(W'^{(\ell)},b'^{(\ell)})_{\ell=1}^L$.
Notice that 
if $\|W^{(\ell)} - W'^{(\ell)}\|_{\F} \leq \epsilon$ and 
$\|b^{(\ell)} - b'^{(\ell)}\| \leq \epsilon$, then
\begin{align}
\|f  - f' \|_\infty
\leq L \epsilon \Rbar^{L-1} D_x + \sum_{\ell = 1}^L \epsilon \Rbar^{L - \ell}
= \epsilon (L \Rbar^{L-1} D_x + \sum_{\ell = 1}^L \Rbar^{L - \ell}).
\label{eq:FunctionMetricToParameterMetric}
\end{align}
Therefore, if $\epsilon \leq \delta/\hat{G}$  
where 
$$\hat{G} = (L \Rbar^{L-1} D_x + \sum_{\ell = 1}^L \Rbar^{L - \ell}),$$
then $\|f  - f' \|_\infty \leq \delta$.
Hence, 
the packing number of the function space $\calF$ can be bounded by using that of the parameter space as 
\begin{align}
&\log(D( r'/2 ,\calF,\sqrt{n}\|\cdot\|_n/\sigma)) 
= \log(D( r'/2 ,\calF,\sqrt{n}\|\cdot\|_n/\sigma)) 
\leq 
  \log(D( \sigma r'/(2\sqrt{n}) ,\calF, \|\cdot\|_\infty)) \notag \\
& 
\leq 
  \log(N( \sigma r'/(2\sqrt{n}) ,\calF, \|\cdot\|_\infty)) \notag \\
&
\leq
\sum_{\ell=1}^L \log(N(\sigma r'/(2\sqrt{n} \hat{G}), \calB_{m_{\ell +1} \times m_{\ell}}(\Rbar),\|\cdot\|))
+
\sum_{\ell=1}^L \log(N(\sigma r'/(2\sqrt{n} \hat{G}), \calB_{m_{\ell}}(\Rbarb),\|\cdot\|)) \notag \\
&
\leq 
\sum_{\ell=1}^L m_{\ell +1} m_\ell \log\left( \frac{4 + \frac{\sigma r'}{2\sqrt{n} \hat{G}\Rbar} }{\frac{\sigma r'}{2\sqrt{n}\hat{G}\Rbar}} \right)
+
\sum_{\ell=1}^L \mell \log\left( \frac{4 + \frac{\sigma r'}{2\sqrt{n} \hat{G}\Rbarb} }{\frac{\sigma r'}{2\sqrt{n}\hat{G}\Rbarb}} \right) \notag \\
&
=
\sum_{\ell=1}^L m_{\ell +1} m_\ell \log\left(1 + \frac{4 \sqrt{2} \hat{G}\Rbar }{ \epsilonn r} \right)
+
\sum_{\ell=1}^L \mell \log\left( 1 + \frac{4 \sqrt{2} \hat{G}\Rbarb  }{ \epsilonn r} \right).
\label{eq:DrCoveringBound}
\end{align}

Therefore, 
by \Eqref{eq:testInnerBound}, 
we have that 
$$
A_n \leq 9 \exp\left[ - \frac{1}{4 \sigma^2 }n \epsilonn^2 r^{2} 
+ \sum_{\ell=1}^L m_{\ell +1} m_\ell \log\left(1 + \frac{4 \sqrt{2} \hat{G}\Rbar }{ \epsilonn r} \right)
+
\sum_{\ell=1}^L \mell \log\left( 1 + \frac{4 \sqrt{2} \hat{G}\Rbarb  }{ \epsilonn r} \right) \right].
$$

%

~\\

\noindent {\bf Step 2.}
Here, we evaluate $B_n$.
It can be evaluated by Lemma \ref{lemm:PriorMass} as 
$$
B_n \leq \exp(- n\tilepsilonn^2 \rtil^2/(8\sigma^2))
+
\exp(-n\deltanone^2 (\rtil^2-1)^2/(11 \Rhatinf^2)).
$$

~\\

\noindent {\it \bf Step 3.}
Since $\calF$ is the support of the prior distribution, it is obvious that $C_n = 0$.

~\\

\noindent {\bf Step 4.}
Here, we evaluate $D_n$.
Remind that $D_n$ is defined as  
$$
D_n = \EE_{X_n}\left[ P_{n,\ftrue}[\Pi(f \in \calF : \|f - \ftrue\|_n > \sqrt{2} \epsilon r | Y_{n}) (1-\phi_n) \boldone_{\calA_{\rtil}}] \right].
$$
Define 
$$
\Xi_n(\rtil) := - \log (\Pi(f : \|f - \fstar \|_{\infty} \leq \deltantwo \rtil ))
$$
for $\rtil > 0$.
Then, $D_n$ can be bounded as 
\begin{align*}
D_n 
&
=
\EE_{X_n}\left\{  
P_{n,\ftrue}\left[ 
\frac{ \int_{\calF} \boldone\{f : \|f - \ftrue\|_n > \sqrt{2} \epsilon r\} p_{n,f} \dd \Pi(f) }{\int_{\calF} p_{n,f} \dd \Pi(f)}
 (1-\phi_n) \boldone_{\calA_{\rtil}} \right] \right\} 
\\
&
=
\EE_{X_n}\left\{  
P_{n,\ftrue}\left[ 
\frac{ \int_{\calF} \boldone\{f : \|f - \ftrue\|_n > \sqrt{2} \epsilon r\} \frac{p_{n,f}}{p_{n,\ftrue}} \dd \Pi(f) }{\int_{\calF} \frac{p_{n,f}}{p_{n,\ftrue}} \dd \Pi(f)}
 (1-\phi_n) \boldone_{\calA_{\rtil}} \right] \right\} 
\\
&\leq   \EE_{X_n}\left\{ P_{n,\ftrue}\left[ \int_{f \in \calF : \|f - \ftrue \|_n > \sqrt{2} \epsilon r}  p_{n,f}/p_{n,\ftrue}  \dd  \Pi(f) 
\exp(n \tilepsilonn^2 \rtil^2/\sigma^2  + \Xi_n(\rtil)) (1-\phi_n)\boldone_{\calA_{\rtil}} \right] \right\} \\
&=  
\EE_{X_n}\left\{   \int_{f \in \calF : \|f - \ftrue \|_n > \sqrt{2} \epsilon r}  P_{n,f}[(1-\phi_n)\boldone_{\calA_{\rtil}}] \exp(n \tilepsilonn^2 \rtil^2/\sigma^2  + \Xi_n(\rtil)) 
 \dd  \Pi(f) \right\} \\
&
\leq 
\exp\left( \frac{n \tilepsilonn^2 \rtil^2}{\sigma^2}  + \Xi_n(\rtil) - \frac{n \epsilonn^2 r^2}{4\sigma^2 }   \right).
\end{align*}
By using the relation \eqref{eq:FunctionMetricToParameterMetric}, 
the prior mass $\Xi_n(\rtil)$ can be bounded as 
\begin{align}
\Xi_n(\rtil) & = - \log (\Pi(f : \|f - \fstar \|_{\infty} \leq \deltantwo \rtil ))  \notag \\
& \leq - \log (\Pi(f : \|f - \fstar \|_{\infty} \leq \deltantwo ))  \notag \\
& \leq - \sum_{\ell=1}^L \log (\Pi(\Well{\ell} : \|\Well{\ell} - \Wstarell{\ell} \|_{\F} \leq \deltantwo/\hat{G} )) \notag \\
& ~~~~- \sum_{\ell=1}^L \log (\Pi(\bell{\ell} : \|\bell{\ell} - \bstarell{\ell} \|_{2} \leq \deltantwo/\hat{G} ))  
 \notag \\
& \leq 
\sum_{\ell=1}^L m_\ell m_{\ell + 1} \log(\Rbar \hat{G}/  (\deltantwo/2))
+
\sum_{\ell =1}^L \mell \log(\Rbarb \hat{G}/(\deltantwo/2)).
\label{eq:XinBound}
\end{align}

\noindent {\it \bf Step 5.}
Finally, we combine the results obtained above.

\begin{align}
& \EE\left[ \Pi(\|f - \ftrue \|_n \geq \sqrt{2} \epsilonn r  | Y_{n} )  \right] \notag \\
\leq
&
9 \exp\left[ - \frac{1}{4\sigma^2}n \epsilonn^2 r^{2} 
+ \sum_{\ell=1}^L m_{\ell +1} m_\ell \log\left(1 + \frac{4 \sqrt{2} \hat{G}\Rbar }{ \epsilonn r} \right)
+
\sum_{\ell=1}^L \mell \log\left( 1 + \frac{4 \sqrt{2} \hat{G}\Rbarb  }{ \epsilonn r} \right) \right] 
\notag
\\
& +
\exp(- n\tilepsilonn^2 \rtil^2/(8\sigma^2))
+
\exp(-n\deltanone^2 (\rtil^2-1)^2/(11 \Rhatinf^2)) \notag
\\
& +
\exp\left(\frac{n}{\sigma^2} \tilepsilonn^2 \rtil^2  + \Xi_n(\rtil) - \frac{n \epsilonn^2 r^2}{4\sigma^2 }   \right).
\label{eq:InsampleBoundMatomeOne}
\end{align}
Now, let $1 \leq \rtil \leq r$.
Then, since $\epsilonn \geq \deltantwo$ and $r \geq 1$, 
we have that
$$
\max\left\{\log\left(\frac{2\hat{G} R'}{\deltantwo} \right),
\log\left(1 + \frac{4 \sqrt{2} \hat{G}R'}{\epsilonn r} \right)
\right\}
\leq 
\log\left(1 + \frac{4 \sqrt{2} \hat{G}R'}{\deltantwo} \right),
$$
for all $R' > 0$.
Now, we set $\deltantwo$ to satisfy
\begin{align}
\frac{n \deltantwo^2}{\sigma^2} & 
\geq \sum_{\ell=1}^L \mell m_{\ell +1}
\log\left(1 + \frac{4 \sqrt{2} \hat{G}\Rbar}{\deltantwo} \right)
+
\sum_{\ell=1}^L 
\mell
\log\left(1 + \frac{4 \sqrt{2} \hat{G}\Rbarb}{\deltantwo} \right) ( \geq \Xi_n(\tilde{r})),
\label{eq:deltantwoCondition}
\end{align}
which can be satisfied by 
$$
\deltantwo^2 =
\frac{2 \sigma^2}{n} \sum_{\ell=1}^L \mell m_{\ell +1}
\logone \left(1 + \frac{4 \sqrt{2} \hat{G}\max\{\Rbar,\Rbarb\}
\sqrt{n}}{
\sigma \sqrt{\sum_{\ell=1}^L \mell m_{\ell +1}}} \right).
$$
Then, by noticing 
$n \deltantwo^2\leq n \tilepsilonn^2$
and \Eqref{eq:XinBound},
the RHS of \Eqref{eq:InsampleBoundMatomeOne}
is upper bounded by
$$
\exp(- n\tilepsilonn^2 \rtil^2/(8\sigma^2))
+
\exp(-n\deltanone^2 (\rtil^2-1)^2/(11 \Rhatinf^2)) \\
+
10 
\exp\left[ 2 \frac{n}{\sigma^2} \tilepsilonn^2  \rtil^2 - \frac{n\epsilonn^2 r^2}{4\sigma^2}  \right].
$$
Here, by setting $r^2 = 12 \rtil^2 \geq 12$, then the RHS is further bounded as 
\begin{align*}
& \exp(-n\deltanone^2 (\rtil^2-1)^2/(11 \Rhatinf^2)) 
+ \exp(- n\tilepsilonn^2 \rtil^2/(8\sigma^2))
+ 10 \exp(- n \epsilonn^2 \rtil^2/\sigma^2) \\
\leq 
& \exp\left[-n\deltanone^2 (\rtil^2-1)^2/(11 \Rhatinf^2)\right]
+ 11 \exp(- n \epsilonn^2 \rtil^2/(8\sigma^2)).
\end{align*}

%
%
%
%
%
%

\begin{Lemma}
\label{lemm:PriorMass}

Then, 
for any $\rtil >1$, it holds that
\begin{align*}
& P_{D_n} \left( \int \frac{p_{n,f}(Y_n)}{p_{n,\ftrue}(Y_n)} \Pi(\dd f) \geq \exp(- n \tilepsilonn^2 \rtil^2/\sigma^2) 
\Pi(f : \|f - \fstar \|_{\infty} \leq \deltantwo \rtil )\right) \\
& \geq 1- 
\exp(- n\tilepsilonn^2 \rtil^2/(8\sigma^2))
-
\exp(-n\deltanone^2 (\rtil^2-1)^2/(11 \Rhatinf^2)). 
\end{align*}
\end{Lemma}

\begin{proof}

Note that Lemma 14 of \cite{JMLR:Vaart&Zanten:2011} showed that 
$$
P_{Y_n|X_n} \left( \int \frac{p_{n,f}(Y_n)}{p_{n,\ftrue}(Y_n)} \Pi(\dd f) \geq \exp(- n \tilepsilonn^2 \rtil^2/\sigma^2) 
\Pi(f : \|f - \ftrue \|_{n} \leq \tilepsilonn \rtil )\right)
\geq 1- \exp(- n \tilepsilonn^2 \rtil^2/(8\sigma^2)). 
$$
where $P_{Y_n|X_n}$ represents the conditional distribution of $Y_n = (y_i)_{i=1}^n$ conditioned by $X_n = (x_i)_{i=1}^n$.
Therefore the proof is reduced to show 
$\|f - \ftrue \|_{n} \leq \deltanone \tilde{r} + \|f - \fstar \|_{\infty}$ with high probability.
Note that
\begin{align*}
\|f - \ftrue \|_{n}
\leq \|f - \fstar \|_{n} + \|\fstar - \ftrue \|_{n}
\leq \|f - \fstar \|_{\infty} + \|\fstar - \ftrue \|_{n}.
\end{align*}
Hence, we just need to show $\|\fstar - \ftrue \|_{n} \leq \tilde{r} \|\fstar - \ftrue \|_{\LPi} \leq \tilde{r} \deltanone$ with high probability.
This can be shown by Bernstein's inequality:
%
%
$$
P\left(\sqrt{1 +\tilde{r}' }\|\fstar-\ftrue \|_{\LPi} \leq \|\fstar -\ftrue \|_n \right)
\leq \exp\left(- \frac{n \tilde{r}'^2 \|\fstar-\ftrue \|_{\LPi}^4}{2(v + \|\fstar-\ftrue\|_{\infty}^2\|\fstar-\ftrue \|_{\LPi}^2/3)}  \right),
$$
where $v = \EE_X[ ( (\fstar(X) - \ftrue(X))^2  -\|\fstar - \ftrue \|_{\LPi}^2)^2 ]$.
Now $v \leq \EE_X[  (\fstar(X) - \ftrue(X))^4 ]
\leq \|\fstar-\ftrue \|_\infty^2 \|\fstar - \ftrue\|_{\LPi}^2
=\|\fstar-\ftrue \|_\infty^2 \|\fstar - \ftrue\|_{\LPi}^2$. 
This yields that
\begin{align}
P\left(\sqrt{1 +\tilde{r}' }\|\fstar-\ftrue \|_{\LPi} \leq \|\fstar -\ftrue \|_n \right)
\leq \exp\left[-\frac{3 n \tilde{r}'^2 }{8 }  
\left( \frac{\|\fstar -\ftrue\|_{\LPi}}{ \|\fstar -\ftrue\|_{\infty}} \right)^2 \right].
\label{eq:BernsteinL2Bound}
\end{align}
Since $\| \fstar - \ftrue \|_{\infty} \leq 2 \Rhatinf$ and $\| \fstar - \ftrue\|_{\LPi} \leq \deltanone$, 
the RHS is further bounded by 
$
\exp\left(-\frac{3 n \tilde{r}'^2 \deltanone^2}{32\Rhatinf^2}  \right).
$

Therefore, with probability $1- \exp\left(-  \frac{3 n \deltanone^2 \tilde{r}'^{2}}{ 32 \Rhatinf^2}   \right)$, it holds that 
\begin{align*}
\|f - \ftrue \|_{n}
\leq \|f - \fstar \|_{\infty} + \sqrt{1 + \tilde{r}'} \|\fstar - \ftrue \|_{\LPi}
\leq \|f - \fstar \|_{\infty} + \sqrt{1 + \tilde{r}'} \deltanone
\end{align*}
for all $f$ such that $\|f\|_\infty < \infty$.
Thus by setting $\tilde{r}'$ so that $\rtil = \sqrt{1 +\tilde{r}'}$, we obtain the assertion.
\end{proof}

\subsection{Out of sample error}
\label{suppsec:OutOfPredError}
Now, we are going to show the posterior contraction rate with respect to the out-of-sample predictive error:
\begin{align}
\EE_{D_n}\left[  \Pi(f : \|f - \ftrue \|_{\LPi} \geq \epsilonn r | D_n) \right],
\label{eq:ExOutPredError}
\end{align}
for sufficiently large $r \geq 1$.

To bound the posterior tail, we divide that into four parts:
\begin{align*}
&
\mathrm{I} = \EE_{D_n}\left[ \boldone_{\calA_{\rtil}^c} \right],  \\ 
& 
\mathrm{II} = \EE_{D_n}\left[
 \boldone_{\calA_{\rtil}} \Pi(f: \sqrt{2} \|f - \ftrue \|_n >  \epsilonn r , ~\|f\|_{\infty} \leq  \Rhatinf  \mid D_n)  \right] ,
 \\
&
\mathrm{III}= \EE_{D_n}\left[
\boldone_{\calA_{\rtil}} \Pi(f: \|f - \ftrue \|_{\LPi}  > \epsilonn r \geq \sqrt{2} \|f - \ftrue \|_n,~\|f \|_{\infty} \leq  \Rhatinf  \mid D_n)
\right], \\
&
\mathrm{IV}= \EE_{D_n}\left[
\boldone_{\calA_{\rtil}} \Pi(f:  \|f \|_{\infty} >  \Rhatinf  \mid D_n)
\right].
\end{align*}

The term $\mathrm{I}$ and $\mathrm{II}$ are already evaluated in Section \ref{sec:AppInsampleError},
that is, $\mathrm{I} + \mathrm{II}$
is bounded by the right hand side of \Eqref{eq:decompIntABCD} 
which is what we have upper bounded in Section \ref{sec:AppInsampleError}.

The term $\mathrm{III}$ is bounded as follows.
To bound this, we need to evaluate the difference between the empirical norm $\|f-\ftrue \|_n$ 
and the expected norm $\|f-\ftrue \|_{\LPi}$, which can be done by Bernstein's inequality.
Following the same argument to derive \Eqref{eq:BernsteinL2Bound}, 
it holds that
$$
P\left(\|f-\ftrue \|_{\LPi} \geq \sqrt{2} \|f-\ftrue\|_n \right)
\leq \exp\left(-\frac{n \|f-\ftrue\|_{\LPi}^2}{11 \Rhatinf^2}  \right).
$$

Therefore, we arrive at the following bound of III:
\begin{align*}
\mathrm{III}
&\leq  \EE_{X_n}\left[  P_{n,\ftrue}\left[ \int_{f \in \calF :\|f - \ftrue \|_{\LPi}  > \epsilonn r \geq \sqrt{2} \|f - \ftrue \|_n }  p_{n,f}/p_{n,\ftrue}  \dd  \Pi(f) \right] 
\exp(n \tilepsilonn^2 \rtil^2/\sigma^2  + \Xi_n(\rtil))  \boldone_{\calA_{\rtil}} 
\right]\\
&\leq 
\exp(n \tilepsilonn^2 \rtil^2/\sigma^2  + \Xi_n(\rtil)) 
 \int_{f \in \calF: \|f - \ftrue \|_{\LPi}  > \epsilonn r}   
P(\|f - \ftrue \|_{\LPi}  \geq \sqrt{2} \|f - \ftrue \|_n)
 \dd  \Pi(f) \\
&
\leq \exp\left( \frac{n \tilepsilonn^2 \rtil^2}{\sigma^2}  + \Xi_n(\rtil)  -\frac{n \epsilonn^2 r^2 }{11 \Rhatinf^2}  \right) \\
&
\leq 
 \exp\left( \frac{2 n \tilepsilonn^2 \rtil^2}{\sigma^2} -\frac{n \epsilonn^2 r^2 }{11 \Rhatinf^2}  \right).
\end{align*}
%
%

Finally, since all $f \in \calF$ satisfies $\|f\|_\infty \leq \Rhatinf$, $\mathrm{IV} = 0$.

Combining the results we arrive at 
\begin{align*}
\EE_{D_n}\left[  \Pi(f : \|f - \ftrue \|_{\LPi} \geq \epsilonn r | D_n) \right]
\leq \exp\left[-n\deltanone^2 (\rtil^2-1)^2/(11 \Rhatinf^2)\right] + 12 \exp\left(  - n \tilepsilonn^2 \rtil^2/(8\sigma^2) \right),
\end{align*}
for all $\rtil \geq 1$ and $r  \geq \max\{ 12, 33 \Rhatinf^2/\sigma^2 \} \rtil^2 $.
This concludes the proof of Theorem \ref{th:PosteriorContraction}.


\section{Convergence rate for the empirical risk minimizer}

\label{sec:ERMproof}

\begin{Proposition}[Gaussian concentration inequality (Theorem 2.5.8 in \cite{GineNickl2015mathematical})]
Let $(\xi_i)_{i=1}^n$ be i.i.d. Gaussian sequence with mean 0 and variance $\sigma^2$,
and $(x_i)_{i=1}^n \subset \calX$ be a given set of input variables.
Then, for a set $\tilde{\calF}$ of functions  from $\calX$ to $\Real$ which is 
separable with respect to $L_\infty$-norm and $\sup_{f \in \tilde{\calF}} \left| \sum_{i=1}^n  \frac{1}{n} \xi_i f(x_i) \right| < \infty$ almost surely,
it holds that for every $r > 0$, 
$$
P\left( \sup_{f \in \tilde{\calF}} \left| \sum_{i=1}^n  \frac{1}{n} \xi_i f(x_i) \right| \geq \EE \left[\sup_{f \in \tilde{\calF}}  
\left| \frac{1}{n} \sum_{i=1}^n \xi_i f(x_i) \right|\right] + r \right) \leq 
\exp[- n r^2/2(\sigma \|\tilde{\calF}\|_n)^2]
$$
where $\|\tilde{\calF}\|_n^2 = \sup_{f \in \tilde{\calF}} \frac{1}{n} \sum_{i=1}^n f(x_i)^2$.
Here the probability is taken with respect to $(\xi_i)_{i=1}^n$.
\end{Proposition}


Remind that every $f \in \calF$ satisfies $\|f\|_n \leq \|f\|_\infty \leq \Rhatinf$.
Hence $\|\calF\|_n \leq \Rhatinf$.
For an observation $(x_i)_{i=1}^n$, 
let $\calG_\delta = \{f - \fstar \mid \|f - \fstar\|_n \leq \delta, f \in \calF\}$. 
It is obvious that $\calG_\delta$ is separable with respect to $L_\infty$-norm.
Then, by the Gaussian concentration inequality, we have that 
$$
P\left( \sup_{f \in \calG_\delta } \left| \sum_{i=1}^n  \frac{1}{n} \xi_i f(x_i) \right| \geq \EE \left[\sup_{f \in \calG_\delta}  
\left|\frac{1}{n} \sum_{i=1}^n \xi_i f(x_i) \right| \right] + r \right) \leq 
\exp[- n r^2/2(\sigma \delta)^2]
$$
for every $r > 0$.
By applying this inequality for $\delta_j = 2^{j-1}\sigma/\sqrt{n}$ for $j=1,\dots, \lceil \log_2(\Rhatinf \sqrt{n}/\sigma) \rceil$
and using the uniform bound,
we can show that, for every $r > 0$, with probability $\lceil \log_2(\Rhatinf \sqrt{n}/\sigma) \rceil \exp[- n r^2/2\sigma^2]$,
it holds that 
any $f \in \calG_\delta$ uniformly satisfies 
$$
\left|  \frac{1}{n} \sum_{i=1}^n  \xi_i (f(x_i) -\fstar(x_i))  \right|
\geq 
\EE \left[\left| \sup_{f \in \calG_{2\delta}}  \frac{1}{n} \sum_{i=1}^n \xi_i f(x_i) \right| \right]
+ 2 \delta  r 
$$
where $\delta$ is any positive real satisfying $\delta \geq \sigma/\sqrt{n}$ and $f \in \calG_\delta$.

\begin{Lemma}
There exists a universal constant $C$ such that 
for any $\delta$ it holds that
$$
\EE \left[ \left| \sup_{f \in \calG_{2\delta}}  \frac{1}{n} \sum_{i=1}^n \xi_i f(x_i) \right| \right]
\leq C \sigma \delta  \sqrt{\frac{\sum_{\ell=1}^L m_\ell m_{\ell+1}}{n}  
\log_+ \left(1 + \frac{4 \hat{G}\max\{\Rbar,\Rbarb\} }{\delta} \right)}.
$$
\end{Lemma}

\begin{proof}
Since $f \mapsto \frac{1}{\sqrt{n}} \sum_{i=1}^n \xi_i f(x_i)$ is a sub-Gaussian process relative to the metric $\|\cdot\|_n$.
By the chaining argument (see, for example, Theorem 2.3.6 of \cite{GineNickl2015mathematical}), it holds that
$$
\EE \left[ \left| \sup_{f \in \calG_{2\delta}}  \frac{1}{n} \sum_{i=1}^n \xi_i f(x_i) \right| \right]
\leq
4\sqrt{2}  \frac{\sigma}{\sqrt{n}} \int_0^{2 \delta} \sqrt{\log(2 N(\epsilon,\calG_{2\delta},\|\cdot\|_n))} \dd \epsilon.
$$
Since $\log N(\epsilon,\calG_{2\delta},\|\cdot\|_n) \leq \log  N(\epsilon,\calF,\|\cdot\|_\infty) \leq 
2 \frac{\sum_{\ell=1}^L m_\ell m_{\ell+1}}{n} 
\log\left(1 + \frac{4 \hat{G}\max\{\Rbar,\Rbarb\} }{\epsilon} \right)
$, the right hand side is bounded by
\begin{align*}
\int_0^{2 \delta} \sqrt{\log(2 N(\epsilon,\calF,\|\cdot\|_n))} \dd \epsilon
& 
\leq 
\int_0^{2 \delta} \sqrt{\log(2) + 2 \frac{\sum_{\ell=1}^L m_\ell m_{\ell+1}}{n} 
\log\left(1 + \frac{4 \hat{G}\max\{\Rbar,\Rbarb\} }{\epsilon} \right) } \dd \epsilon \\
& 
\leq 
C \delta \sqrt{ \frac{\sum_{\ell=1}^L m_\ell m_{\ell+1}}{n} 
\log_+\left(1 + \frac{4 \hat{G}\max\{\Rbar,\Rbarb\} }{\delta} \right) },
\end{align*}
where $C$ is a universal constant.
This gives the assertion.

\end{proof}

Therefore, 
by substituting 
$\delta \leftarrow \left (\|f - \fstar\|_n \vee \sigma \sqrt{\frac{\sum_{\ell=1}^L m_\ell m_{\ell + 1}}{n}}\right)$
and $r \leftarrow \sigma r/\sqrt{n}$, 
the following inequality holds:
\begin{align*}
&- \frac{1}{n} \sum_{i=1}^n  \xi_i (f(x_i) -\fstar(x_i)) 
 \\
&
\leq 
C \sigma\left (\|f - \fstar\|_n \vee \sqrt{\frac{\sigma^2 \sum_{\ell=1}^L m_\ell m_{\ell+1}}{n}} \right)  
\sqrt{\frac{\sum_{\ell=1}^L m_\ell m_{\ell+1}}{n}  
\log_+ \left(1 + \frac{4 \sqrt{n} \hat{G}\max\{\Rbar,\Rbarb\}}{\sigma\sqrt{\sum_{\ell=1}^L m_\ell m_{\ell+1}}}   \right)} \\
& + 2\left (\|f - \fstar\|_n \vee 
\sqrt{\frac{\sigma^2 \sum_{\ell=1}^L m_\ell m_{\ell+1}}{n}}\right) \sigma \frac{r}{\sqrt{n}} \\
&
\leq 
\frac{1}{4} \left (\|f - \fstar\|_n \vee \sqrt{\frac{\sigma^2 \sum_{\ell=1}^L m_\ell m_{\ell+1}}{n}}\right)^2 \\
& + 2 C^2 \sigma^2 
\left( \frac{\sum_{\ell=1}^L m_\ell m_{\ell+1}}{n}  
\log_+ \left(1 + \frac{4 \sqrt{n} \hat{G}\max\{\Rbar,\Rbarb\}}{\sigma} \right)
+ 4  \frac{r^2}{n} \right),
\end{align*}
uniformly for all $f\in \calF$ with probability $1-\lceil \log_2(\Rhatinf \sqrt{n}/\sigma) \rceil \exp[- r^2/2]$.
Here let 
$$
\Psi_{r,n} :=2  C^2 \sigma^2 
\left( \frac{\sum_{\ell=1}^L m_\ell m_{\ell+1}}{n}  
\log_+ \left(1 + \frac{4 \sqrt{n} \hat{G}\max\{\Rbar,\Rbarb\}}{\sigma \sqrt{\sum_{\ell=1}^L m_\ell m_{\ell + 1} }} \right)
+ 4  \frac{r^2}{n} \right).
$$

Remind that the empirical risk minimizer in the model $\calF$ is denoted by $\fhat$:
$$
\fhat := \argmin_{f \in \calF} \sum_{i=1}^n (y_i - f(x_i))^2.
$$
Since $\fhat$ minimizes the empirical risk, it holds that
\begin{align*}
& 
\frac{1}{n} \sum_{i=1}^n (y_i - \fhat(x_i))^2 \leq
\frac{1}{n} \sum_{i=1}^n (y_i - \fstar(x_i))^2 \\
\Rightarrow~~ & 
\frac{2}{n} \sum_{i=1}^n y_i (\fstar(x_i) - \fhat(x_i)) 
+ \|\fhat\|_n^2 - \|\fstar\|_n^2 \leq 0 \\
\Rightarrow~~
& 
\frac{2}{n} \sum_{i=1}^n (\xi_i + \ftrue(x_i)) (\fstar(x_i) - \fhat(x_i)) 
+ \|\fhat\|_n^2 - \|\fstar\|_n^2 \leq 0 \\
\Rightarrow~~
& 
\frac{2}{n} \sum_{i=1}^n \xi_i  (\fstar(x_i) - \fhat(x_i)) 
+
\frac{2}{n} \sum_{i=1}^n \ftrue(x_i) (\fstar(x_i) - \fhat(x_i)) 
+ \|\fhat\|_n^2 - \|\fstar\|_n^2 \leq 0 \\
\Rightarrow~~
& 
\frac{2}{n} \sum_{i=1}^n \xi_i  (\fstar(x_i) - \fhat(x_i)) 
+ \|\fhat - \ftrue \|_n^2  \leq \|\fstar - \ftrue\|_n^2.
\end{align*}
Therefore, we have 
\begin{align}
& - \frac{1}{4} \left (\|\fhat - \fstar\|_n \vee \sqrt{\frac{\sigma^2 \sum_{\ell=1}^L m_\ell m_{\ell+1}}{n}} \right)^2
- \Psi_{r,n} + \|\fhat - \ftrue\|_n^2 \leq \|\fstar - \ftrue\|_n^2.
%
\label{eq:fhatfstarFirstbound}
\end{align}

%

Let us assume  $\|\fhat - \fstar\|_n^2 \geq \frac{\sigma^2 \sum_{\ell=1}^L m_\ell m_{\ell+1}}{n}$.
Then, by \Eqref{eq:fhatfstarFirstbound}, we have 
%
\begin{align}
& - \frac{1}{4} \|\fhat - \fstar\|_n^2
- \Psi_{r,n} + \|\fhat - \ftrue\|_n^2 \leq \|\fstar - \ftrue\|_n^2 \notag \\
\Rightarrow~~
& - \frac{1}{4} \|\fhat - \fstar\|_n^2
- \Psi_{r,n} + \frac{1}{2} \|\fhat - \fstar\|_n^2 - \|\fstar - \ftrue \|_n^2 \leq \|\fstar - \ftrue\|_n^2 \notag \\
\Rightarrow~~
&
\frac{1}{4} \|\fhat - \fstar\|_n^2 \leq 
2 \|\fstar - \ftrue\|_n^2 + \Psi_{r,n}.
\end{align}
Otherwise,
we trivially have $\|\fhat - \fstar\|_n^2 < 
\frac{\sigma^2 \sum_{\ell=1}^L m_\ell m_{\ell+1}}{n}
$.

Combining the inequalities, 
it holds that 
\begin{align}
\|\fhat - \fstar\|_n^2 \leq 8  \|\fstar - \ftrue\|_n^2 + 4 \Psi_{r,n} + 
\frac{\sigma^2 \sum_{\ell=1}^L m_\ell m_{\ell+1}}{n}.
\label{eq:fhatEmpL2Bound}
\end{align}
Based on this inequality, we derive a bound for $\|\fhat - \fstar\|_{\LPi}$ instead of the empirical $L_2$-norm $\|\fhat - \fstar\|_n$.

\begin{Proposition}[Talagrand's concentration inequality \citep{Talagrand2,BousquetBenett}]
Let $(x_i)_{i=1}^n$ be an i.i.d. sequence of input variables in $\calX$.
Then, for a set $\tilde{\calF}$ of functions  from $\calX$ to $\Real$ which is 
separable with respect to $L_\infty$-norm and 
$\|f\|_\infty \leq \tilde{R}$ for all $f \in \tilde{\calF}$, 
it holds that for every $r > 0$, 
\begin{align*}
& P\left( \sup_{f \in \tilde{\calF}} \left | \frac{1}{n} \sum_{i=1}^n f(x_i)^2 - \EE[f^2]\right | \geq 
C \left\{
\EE \left[\sup_{f \in \tilde{\calF}} \left | \frac{1}{n} \sum_{i=1}^n f(x_i)^2 - \EE[f^2]\right | \right] + 
\sqrt{\frac{\|\tilde{\calF}^2\|_{\LPi}^2 r}{n}}
+
\frac{r \tilde{R}^2}{n} 
\right\}
\right) \\
&
\leq  
\exp(-r)
\end{align*}
where $\|\tilde{\calF}^2\|_{\LPi}^2 = \sup_{f \in \tilde{\calF}} \EE[ f(X)^4]$.
\end{Proposition}

Let $\calG'_\delta = \{f - \fstar \mid \|f - \fstar\|_{\LPi} \leq \delta, f\in\calF \}$. 
By the bound $\|f\|_\infty \leq \Rhatinf$ for all $f \in \calF$ (Lemma \ref{supplemma:SupnormBounds}), 
$\|g\|_\infty \leq 2 \Rhatinf$ for all $g \in \calG'_\delta$.
Therefore, we have $\|\calG'^{2}_\delta\|_{\LPi}^2 \leq 4 \Rhatinf^2 \delta^2$.
Hence, Talagrand's concentration inequality yields that
\begin{align}
\sup_{f \in \calG'_\delta}  \left | \frac{1}{n} \sum_{i=1}^n f(x_i)^2 - \EE[f^2]\right| \geq 
C_1 \left\{
\EE \left[\sup_{f \in \calG'_\delta}  \left | \frac{1}{n} \sum_{i=1}^n f(x_i)^2 - \EE[f^2]\right|  \right] + 
\sqrt{\frac{\delta^2 \Rhatinf^2 r}{n}}
+
\frac{r \tilde{R}^2}{n} 
\right\}
\label{eq:TalagrandIneqForSquare}
\end{align}
with probability $1 - \exp(-r)$ where $C_1$ is a universal constant. 

\begin{Lemma}

There exists a universal constant $C > 0$ such that,
for all $\delta > 0$, 
\begin{align*}
& \EE \left[\sup_{f \in \calG'_\delta}  \left | \frac{1}{n} \sum_{i=1}^n f(x_i)^2 - \EE[f^2]\right |  \right] \\
& \leq C
\Bigg[ \delta \Rhatinf  \sqrt{\frac{\sum_{\ell=1}^L m_\ell m_{\ell+1}}{n}  
\log_+ \left(1 + \frac{4 \hat{G}\max\{\Rbar,\Rbarb\} }{\delta} \right)} \\
&~~
\vee \Rhatinf^2
\frac{\sum_{\ell=1}^L m_\ell m_{\ell+1}}{n}  
\log_+ \left(1 + \frac{4 \hat{G}\max\{\Rbar,\Rbarb\} }{\delta} \right) \Bigg].
\end{align*}
\end{Lemma}

\begin{proof}

Let $(\epsilon_i)_{i=1}^n$ be i.i.d. Rademacher sequence. 
Then, by the standard argument of Rademacher complexity, we have 
\begin{align*}
& \EE \left[\sup_{f \in \calG'_\delta }  \left | \frac{1}{n} \sum_{i=1}^n f(x_i)^2 - \EE[f^2]\right| \right] 
\leq
2 \EE \left[\sup_{f \in \calG'_\delta}  \left| \frac{1}{n} \sum_{i=1}^n \epsilon_i f(x_i)^2  \right|\right]
\end{align*}
(see, for example, Lemma 2.3.1 in \citet{Book:VanDerVaart:WeakConvergence}).
Since $\|f\|_\infty \leq 2 \Rhatinf$ for all $f \in \calG'_\delta$, 
the contraction inequality \cite[Theorem 4.12]{Book:Ledoux+Talagrand:1991} gives an upper bound of the RHS as 
\begin{align*}
&2 \EE \left[\sup_{f \in \calG'_\delta}  \left| \frac{1}{n} \sum_{i=1}^n \epsilon_i f(x_i)^2 \right| \right] 
\leq
4 (2\Rhatinf )
\EE \left[\sup_{f \in \calG'_\delta}  \left| \frac{1}{n} \sum_{i=1}^n \epsilon_i f(x_i)  \right|\right].
\end{align*}
We further bound the RHS. By Theorem 3.1 in \cite{gine2006concentration} or Lemma 2.3 of \cite{IEEEIT:Mendelson:2002} with the covering number bound \eqref{eq:DrCoveringBound}, 
there exists a universal constant $C'$ such that
\begin{align*}
& \EE \left[\sup_{f \in \calG'_\delta}  
\left| \frac{1}{n} \sum_{i=1}^n \epsilon_i f(x_i)  \right| \right] \\
& \leq  
C' 
\Bigg[ \delta \sqrt{\frac{\sum_{\ell=1}^L m_\ell m_{\ell+1}}{n}  
\log_+ \left(1 + \frac{4 \hat{G}\max\{\Rbar,\Rbarb\} }{\delta} \right)} \\
&
\vee \Rhatinf
\frac{\sum_{\ell=1}^L m_\ell m_{\ell+1}}{n}  
\log_+ \left(1 + \frac{4 \hat{G}\max\{\Rbar,\Rbarb\} }{\delta} \right)
\Bigg].
\end{align*}
This concludes the proof.
\end{proof}

Let $\Phi_{n} := \frac{\sum_{\ell=1}^L m_\ell m_{\ell+1}}{n}  
\log_+ \left(1 + \frac{4 \sqrt{n} \hat{G}\max\{\Rbar,\Rbarb\}}{\Rhatinf \sqrt{\sum_{\ell=1}^L m_\ell m_{\ell+1}}}  \right)$. Then, applying the inequality \eqref{eq:TalagrandIneqForSquare} for $\delta = 2^{j-1}\Rhatinf /\sqrt{n}$ for $j=1,\dots,\lceil \log_2(\sqrt{n}) \rceil$, 
it is shown that 
there exists an event with probability $1 - \lceil \log_2(\sqrt{n}) \rceil \exp(-r)$
such that, 
uniformly for all $f \in \calF$, it holds that
\begin{align*}
\left| \frac{1}{n} \sum_{i=1}^n (f(x_i) - \fstar(x_i))^2 - \EE[(f - \fstar)^2] \right|
& \leq C_1 \left[ C (2 \delta \Rhatinf \sqrt{\Phi_{n}}) \vee (\Rhatinf^2 \Phi_{n}) + \delta \sqrt{\frac{\Rhatinf^2 r}{n}} + \frac{r \Rhatinf^2}{n}\right]  \\
& \leq \frac{\delta^2}{2}
+ 2 C_1^2 (2 C^2  +1) \Rhatinf^2 \Phi_{n} +  (C_1^2 + C_1) \frac{\Rhatinf^2 r}{n},
\end{align*}
where $\delta$ is any positive real such that $\delta^2  \geq \EE[(f - \fstar)^2]$ and $\delta^2 \geq \Rhatinf^2 \sum_{\ell=1}^L \mell m_{\ell + 1}/n$.
The right hand side can be further bounded by 
$$
\frac{\delta^2}{2} + C_2  \Rhatinf^2 \left( \Phi_{n}  + \frac{r}{n}\right)
$$
for an appropriately defined universal constant $C_2$.
Applying this inequality for $f = \fhat$ to \Eqref{eq:fhatEmpL2Bound} gives that 
\begin{align*}
\frac{1}{2}\|\fhat - \fstar\|_{\LPi}^2 
\leq C_2  \Rhatinf^2 \left( \Phi_{n}  + \frac{r}{n}\right) + 8 \|\fstar - \ftrue\|_n^2 + 4 \Psi_{r,n} + \left(\frac{\sigma^2 + \Rhatinf^2}{n} \right) \sum_{\ell=1}^L m_\ell m_{\ell + 1}.
\end{align*}
Finally, by the Bernstein's inequality \eqref{eq:BernsteinL2Bound}, the term $\|\fstar - \ftrue\|_n^2$ is bounded as 
$$
\|\fstar - \ftrue\|_n^2 \leq (1 + \tilde{r}')\|\fstar - \ftrue\|_{\LPi}^2  \leq  (1 + \tilde{r}')\deltanone^2
$$
with probability $1- \exp\left(-  \frac{3 n \deltanone^2{\tilde{r}'^2}}{ 32 \Rhatinf^2}   \right)$ for every $\tilde{r}' > 0$.

Combining all inequalities, we obtain that 
\begin{align*}
\|\fhat - \fstar\|_{\LPi}^2 
\leq 2 C_2  \Rhatinf^2 \left( \Phi_{n}  + \frac{r}{n}\right) + 16(1 + \tilde{r}')\deltanone^2 + 4 \Psi_{r,n} + \frac{2(\sigma^2 + \Rhatinf^2)}{n}\sum_{\ell=1}^L m_\ell m_{\ell + 1}.
\end{align*}
This gives a bound for the distance between $\fhat$ and $\fstar$. However, what we want is 
a bound on the distance from the true function $\ftrue$ to $\fhat$.
This can be accomplished 
by noticing that 
$ 
\|\fhat - \ftrue\|_{\LPi}^2 \leq 2(\|\fhat - \fstar\|_{\LPi}^2 + \|\ftrue - \fstar\|_{\LPi}^2)\leq  2 \|\fhat - \fstar\|_{\LPi}^2 + 2 \deltanone^2,
$
and conclude that 
\begin{align*}
\|\fhat - \ftrue\|_{\LPi}^2 \leq
4 C_2  \Rhatinf^2 \left( \Phi_{n}  + \frac{r}{n}\right) + (34 + 32\tilde{r}')\deltanone^2 + 8 \Psi_{r,n} + \frac{4(\sigma^2+\Rhatinf^2)}{n}\sum_{\ell=1}^L m_\ell m_{\ell + 1}.
\end{align*}
More concisely, letting 
$$
\alpha(U) := U^2 \frac{\sum_{\ell=1}^L m_\ell m_{\ell+1}}{ n}
\log_+ \left(1 + {\textstyle \frac{4 \sqrt{n} \hat{G}\max\{\Rbar,\Rbarb\}}{U\sqrt{\sum_{\ell=1}^L m_\ell m_{\ell+1}}}  } \right),
$$
the right side is further upper bounded as 
\begin{align*}
\|\fhat - \ftrue\|_{\LPi}^2 \leq 
& 
C_3
\Bigg\{
\alpha(\Rhatinf) + \alpha(\sigma)
+
\frac{ (\Rhatinf^2 + \sigma^2)}{n}\left[
\log_+ \left(\frac{\sqrt{n}}{\min\{\sigma/\Rhatinf,1\}}\right)  + r
 \right] 
+(1+\tilde{r}')\deltanone^2 
 \Bigg\}
\end{align*}
with probability $1-  \exp\left(-  \frac{3 n \deltanone^2{\tilde{r}'^2}}{ 32 \Rhatinf^2}   \right) - 2\exp(- r)$
for every $r >0$ and $\tilde{r}' > 0$.


\bibliographystyle{abbrvnat}
\bibliography{main}

\end{document}